\documentclass{amsart}
\usepackage[a4paper, margin=1in]{geometry}
\usepackage{amsthm}
\usepackage{aliascnt}
\usepackage{amssymb}
\usepackage[hyphens]{url} 
\usepackage{graphicx}
\usepackage{amsmath}
\usepackage{cite}
\usepackage{mathtools}
\usepackage{mathrsfs}
\usepackage[title]{appendix} 
\usepackage[utf8]{inputenc}
\usepackage[english]{babel}
\usepackage{hyperref}
\usepackage[capitalise,noabbrev]{cleveref}
\RequirePackage{textgreek}
\usepackage[all]{xy}
\usepackage{enumitem}
\usepackage{bm}
\usepackage{nicefrac}
\usepackage{mathabx}
\usepackage{tikz-cd}
\usepackage{booktabs}
\usepackage{tabularx}
\newcolumntype{Y}{>{\raggedright\arraybackslash}X}

\providecommand{\definitionname}{Definition}

\providecommand{\lemmaname}{Lemma}
\providecommand{\questionname}{Question}
\providecommand{\theoremname}{Theorem}

\hypersetup{colorlinks=true,linkcolor=blue,citecolor=cyan,urlcolor=red}
\hypersetup{linktoc=all}

\numberwithin{equation}{subsection}

\newtheorem{thm}{\protect\theoremname}[section]
\theoremstyle{plain}
\newtheorem{mainthm}{Theorem}
\newaliascnt{lem}{thm}
\newtheorem{lem}[lem]{\protect\lemmaname}
\aliascntresetthe{lem}
\newaliascnt{prop}{thm}
\newtheorem{prop}[prop]{Proposition}
\aliascntresetthe{prop}
\newaliascnt{cor}{thm}
\newtheorem{cor}[cor]{Corollary}
\aliascntresetthe{cor}
\newaliascnt{obs}{thm}

\aliascntresetthe{obs}
\theoremstyle{definition}
\newaliascnt{defn}{thm}
\newtheorem{defn}[defn]{\protect\definitionname}
\aliascntresetthe{defn}
\newaliascnt{exs}{thm}

\aliascntresetthe{exs}
\newaliascnt{exm}{thm}
\newtheorem{exm}[exm]{Example}
\aliascntresetthe{exm}
\newaliascnt{question}{thm}

\aliascntresetthe{question}
\newaliascnt{con}{thm}

\aliascntresetthe{con}
\newaliascnt{que}{thm}

\aliascntresetthe{que}
\newaliascnt{rem}{thm}
\newtheorem{rem}[rem]{Remark}
\aliascntresetthe{rem}
\newaliascnt{fct}{thm}
\newtheorem{fct}[fct]{Fact}
\aliascntresetthe{fct}
\newtheorem*{BLP}{\(\quad\)The Borel lifting problem}
\theoremstyle{plain}
\newaliascnt{claim}{thm}
\newtheorem{claim}[claim]{Claim}
\aliascntresetthe{claim}

\crefname{thm}{theorem}{theorems}
\Crefname{thm}{Theorem}{Theorems}
\crefname{mainthm}{theorem}{theorems}
\Crefname{mainthm}{Theorem}{Theorems}
\crefname{lem}{lemma}{lemmas}
\Crefname{lem}{Lemma}{Lemmas}
\crefname{prop}{proposition}{propositions}
\Crefname{prop}{Proposition}{Propositions}
\crefname{cor}{corollary}{corollaries}
\Crefname{cor}{Corollary}{Corollaries}
\crefname{obs}{observation}{observations}
\Crefname{obs}{Observation}{Observations}
\crefname{defn}{definition}{definitions}
\Crefname{defn}{Definition}{Definitions}
\crefname{exs}{examples}{examples}
\Crefname{exs}{Examples}{Examples}
\crefname{exm}{example}{examples}
\Crefname{exm}{Example}{Examples}
\crefname{question}{question}{questions}
\Crefname{question}{Question}{Questions}
\crefname{que}{question}{questions}
\Crefname{que}{Question}{Questions}
\crefname{con}{construction}{constructions}
\Crefname{con}{Construction}{Constructions}
\crefname{rem}{remark}{remarks}
\Crefname{rem}{Remark}{Remarks}
\crefname{fct}{fact}{facts}
\Crefname{fct}{Fact}{Facts}
\crefname{claim}{claim}{claims}
\Crefname{claim}{Claim}{Claims}
\crefname{section}{section}{sections}
\Crefname{section}{Section}{Sections}
\crefname{subsection}{section}{sections}
\Crefname{subsection}{Section}{Sections}
\crefname{subsubsection}{section}{sections}
\Crefname{subsubsection}{Section}{Sections}
\crefname{equation}{equation}{equations}
\Crefname{equation}{Equation}{Equations}
\crefname{figure}{figure}{figures}
\Crefname{figure}{Figure}{Figures}
\crefname{table}{table}{tables}
\Crefname{table}{Table}{Tables}
\crefname{enumi}{item}{items}
\Crefname{enumi}{Item}{Items}
\crefname{enumii}{item}{items}
\Crefname{enumii}{Item}{Items}
\crefname{appendix}{appendix}{appendices}
\Crefname{appendix}{Appendix}{Appendices}

\makeatletter
\def\@settitle{%
  \begin{center}%
    \baselineskip18\p@\relax
    \normalfont\Large\scshape
    \@title
  \end{center}%
}
\makeatother

\author{Nachi Avraham-Re'em}
\address{Department of Mathematics, Technion, Israel}
\email{nachi.avraham@gmail.com}

\author{George Peterzil}
\address{Einstein Institute of Mathematics, The Hebrew University, Israel}
\email{george.peterzil@mail.huji.ac.il}

\thanks{The research was supported by the Wallenberg Foundation (KAW 2021.0258) and the ISF (grants No. 1180/22 and 957/25 and 3423/24). It was initiated during the authors' visit to Br\"ann\"o Island in the southern Gothenburg archipelago.}

\date{\today}

\title{Wasserstein stability and the nonsingular Borel lifting problem}



\begin{document}

\dedicatory{In memory of Johanna Steinmeyer (1995--2026)}

\maketitle

\begin{abstract}
The Borel lifting problem asks when a nonsingular near action of a Polish group can be represented by a genuine Borel action. For locally compact Polish groups, a classical theorem of Mackey and Ramsay gives an affirmative answer. At the opposite extreme, Glasner--Tsirelson--Weiss proved that every probability preserving Borel action of a L\'{e}vy group is trivial, and asked whether a L\'{e}vy group can admit a nontrivial nonsingular Borel action. We prove a fixed-point theorem which gives a negative answer for most of the standard L\'{e}vy groups in the literature: if \(G\) is a locally Wasserstein group, then every quasi-invariant \(\sigma\)-finite Borel measure for a Borel action of \(G\) is supported on the fixed points.

\smallskip

The proof introduces Wasserstein stability, a local-to-global principle for compact measure metric groups that refines the Gromov--Milman framework for concentration of measure in topological groups. It asserts that for a compact measure metric group, on average, the total variation of locally distorted density kernels on the group controls their global Wasserstein distance from the Haar measure. We show that Wasserstein stability implies concentration, and prove a functional inequality using martingales in the spirit of Milman--Schechtman, which is governed by \(\ell^{1}\)-sums of martingale increments. This gives a geometric criterion by which most of the standard L\'{e}vy groups in the literature are Wasserstein or locally Wasserstein, including \(L^{0}\)-groups with compact targets, measure preserving and nonsingular automorphism groups, full groups of amenable equivalence relations, isometry groups of \(L^{p}\)-spaces for \(p\neq2\), the unitary group of the hyperfinite \(\mathrm{II}_{1}\)-factor, the infinite dimensional unitary and orthogonal groups, the Cameron--Martin affine group, and the isometry group of the Urysohn space.
\end{abstract}

\tableofcontents

\section{Introduction}

The \emph{Borel lifting problem}, or the \emph{point-realization problem}, asks whether a \emph{near action} of a group on a measure space can be represented by an actual pointwise Borel action on the underlying space. More explicitly, given a homomorphism from a group into the group of measure class preserving automorphisms modulo null sets, one asks whether there is a genuine Borel action whose point maps represent the prescribed nonsingular automorphisms. For a precise formulation see \cref{sct:nonsingspat}.

\smallskip

This problem has a classical positive solution in the locally compact setting. For probability preserving \(\mathbb{R}\)-flows this goes back to von Neumann \cite{neumann1932operator}, and for nonsingular actions of locally compact Polish groups it was proved by Mackey and Ramsay that a Borel lifting can always be found \cite{mackey1962point,ramsay1971virtual}; see also \cite{glasner2005spatial,danilenko2000,kwiatkowska2011,tornquist2011pointwise}. However, the Borel lifting problem becomes most delicate for more general Polish groups.

\smallskip

This viewpoint was emphasized about 25 years ago, when H. Becker wrote an unpublished note (see \cite[p.~150]{Pestov2006}), where he constructed an example showing that the Borel lifting problem can fail for probability preserving near actions of the group of \(\mathbb{Z}/2\mathbb{Z}\)-valued random variables,
\[L^{0}\left(\left[0,1\right],\mathbb{Z}/2\mathbb{Z}\right).\]
This Polish group is a central example studied by Furstenberg--Weiss and Glasner~\cite[Thm.~1.3]{glasner1998minimal}. Shortly after, Glasner--Tsirelson--Weiss published an influential work where they gave a conceptual explanation by relating the absence of Borel liftings to the concentration of measure phenomenon in topological groups, as formulated by Gromov--Milman \cite{milman1971new,gromov1983,milman1988heritage}. They proved that for probability preserving actions of L\'{e}vy groups, the obstruction is as strong as possible: every probability preserving Borel action of a L\'{e}vy group is trivial \cite{glasner2005automorphism}. This vastly generalizes Becker's example, since \(L^{0}\left(\left[0,1\right],S^{1}\right)\) is a L\'{e}vy group.\footnote{It seems that at the time of writing their paper, Glasner--Tsirelson--Weiss were unaware of Becker's note.}

\smallskip

The result of Glasner--Tsirelson--Weiss is for probability preserving actions. The general case of nonsingular actions, where the action only preserves the measure class, remained open in their work:

\smallskip

\begin{center}
\begin{minipage}{0.94\textwidth}
    \emph{Can a L\'{e}vy group admit a nontrivial nonsingular (that is, preserving a measure class) Borel action?}\\
    \hfill -- {\footnotesize\bfseries Glasner--Tsirelson--Weiss
    \cite[Question 1.2]{glasner2005automorphism}}
\end{minipage}
\end{center}

\smallskip

This question was reiterated by Pestov in \cite[Open Problem~7.1.19]{Pestov2006} and \cite[Question~21]{pestov2007forty}. Recent progress was made by E. Roy and one of the authors~\cite[Thm.~5]{avraham2025poissonian}, and independently by Hoareau--Le Ma{\^\i}tre~\cite[Thm.~E]{hoareau2026}, who showed that under an additional topological regularity assumption, there is a spatial version of the Poisson suspension functor. This functor allows for a reduction to the Glasner--Tsirelson--Weiss theorem.

\smallskip

The main result of this work answers this question for most of the standard L\'{e}vy groups in the literature \cite{glasner2005automorphism,Pestov2006,giordano2007some} as we shall explain below. To this end we develop the notion of \emph{Wasserstein stability} for compact measure metric groups, which refines the Gromov--Milman concentration framework in a way adapted to nonsingular dynamics.

\begin{mainthm}\label{mthm1}
Let \(G\) be a locally Wasserstein group and let \(X\) be a standard Borel \(G\)-space. Then every \(G\)-quasi-invariant \(\sigma\)-finite Borel measure on \(X\) is supported on the set of fixed points.
\end{mainthm}

The main technical tool for proving Wasserstein stability is a new martingale functional inequality. It is analogous in spirit to the Milman--Schechtman concentration inequality \cite[I, Thm.~7.12]{milman1986asymptotic}, but the estimate is governed by an \(\ell^{1}\)-sum of martingale increments rather than an \(\ell^{2}\)-sum. It yields the following purely geometric criterion, which unifies and strengthens several proofs of the L\'{e}vy property in the literature, and gives the stronger Wasserstein property; see \cref{sct:stabineq}.

\begin{mainthm}\label{mthm2}
Let \(G\) be a Polish group with a right-invariant compatible metric \(d\). Suppose there is a sequence of compact subgroups \(K_{1}\leq K_{2}\leq\dotsm\leq G\) whose union is dense in \(G\), and suppose that each \(K_{n}\) admits a finite chain of compact subgroups
\[\{e\}=H_{n,0}\leq H_{n,1}\leq\dotsm\leq H_{n,M_{n}}=K_{n},\]
such that the quotient radii \(\rho_{n,i}=\operatorname{rad}\left(H_{n,i}/H_{n,i-1}\right)\) with respect to \(d\mid_{K_{n}}\) satisfy
\[\sup\nolimits_{n\geq 1}\sum\nolimits_{i=1}^{M_{n}}\rho_{n,i}<+\infty\quad\text{and}\quad\max_{1\leq i\leq M_{n}}\rho_{n,i}\to 0\text{ as }n\to\infty.\]
Then \(G\) is a Wasserstein group (hence also a L\'{e}vy group).
\end{mainthm}

Interestingly, \cref{mthm1} and \cref{mthm2} together yield that a nontrivial locally compact Polish group cannot satisfy the geometric condition in \cref{mthm2}. In \cref{app:geomlcsc} we give a direct explanation to that.

\smallskip

Using \cref{mthm2}, we show in \cref{app:Wass} that the following are Wasserstein groups:
\begin{itemize}
    \item \(L^{0}\left(\left[0,1\right],K\right)\), for \(K\) a compact metrizable group;
    \item \(\mathrm{Aut}\left(\left[0,1\right]\right)\), the group of Lebesgue probability preserving automorphisms;
    \item \(\mathrm{Aut}\left(\mathbb{R}\right)\), the group of Lebesgue infinite measure preserving automorphisms;
    \item \(\mathrm{Aut}^{\ast}\left(\left[0,1\right]\right)\), the group of Lebesgue nonsingular automorphisms;
    \item \(\left[\mathcal{R}\right]\), the full group of a nonsingular ergodic amenable countable Borel equivalence relation;
    \item \(\mathrm{Iso}\left(L^{p}\left(\left[0,1\right]\right)\right)\), for \(1\leq p<\infty\), \(p\neq2\);
    \item \(U\left(R\right)\), the unitary group of the hyperfinite \(\mathrm{II}_{1}\)-factor.
\end{itemize}

We also introduce locally Wasserstein groups, which are Polish groups topologically generated by continuous images of Wasserstein groups. This enlargement is important because several central L\'{e}vy groups are naturally obtained from Wasserstein pieces rather than from one Wasserstein skeleton. In \cref{app:locWass} we prove that the following are locally Wasserstein groups:
\begin{itemize}
    \item \(U\left(\ell^{2}\right)\) and \(O\left(\ell_{\mathbb{R}}^{2}\right)\) with the strong operator topology;
    \item \(\ell_{\mathbb{R}}^{2}\rtimes O\left(\ell_{\mathbb{R}}^{2}\right)\), the Cameron--Martin affine group;
    \item \(\mathrm{Iso}\left(\mathbb{U}\right)\), the isometry group of the Urysohn space;
    \item \(\left[\mathcal{R}\right]\), the full group of a nonsingular countable Borel equivalence relation (possibly nonamenable and non-ergodic).
\end{itemize}

The example of \(\left[\mathcal{R}\right]\) is noteworthy from a theoretical perspective, because it is not a L\'{e}vy group as long as \(\mathcal{R}\) is nonamenable, by a theorem of Giordano--Pestov~\cite[Thm.~5.7]{giordano2007some}, and still is a locally Wasserstein group; see \cref{cor:locWassnonLev}. In particular, it constitutes a nonsingular analog to Pestov's answer \cite{Pestov2010} of a problem of Glasner--Weiss~\cite[Problem~2]{glasner2005spatial} on the absence of Borel liftings beyond L\'{e}vy groups.

\smallskip

The Cameron--Martin example is particularly relevant to the original motivation. The standard Gaussian near action of \(O\left(\ell_{\mathbb{R}}^{2}\right)\), considered by Glasner--Tsirelson--Weiss, extends by the Cameron--Martin theorem to a nonsingular near action of \(\ell_{\mathbb{R}}^{2}\rtimes O\left(\ell_{\mathbb{R}}^{2}\right)\). Thus \cref{mthm1} gives a nonsingular analogue of the Glasner--Tsirelson--Weiss fixed-point phenomenon for this natural Gaussian setting; see \cref{exm:CamMar}.

\smallskip

It is worth mentioning that there are Polish groups admitting no nontrivial strongly continuous unitary representations, so called \emph{exotic groups}; see \cite[Rem.~1.7]{glasner2005automorphism}, \cite[pp.~156-157]{Pestov2006}, \cite{schneider2025groups} and the references therein. Such groups admit no nontrivial nonsingular near actions, and a fortiori no nontrivial nonsingular Borel actions. Locally Wasserstein groups form a different class: they may admit nontrivial nonsingular near actions, but no nontrivial nonsingular Borel actions.

\subsection{About the proof}

Let us explain why the nonsingular Borel lifting problem requires more than ordinary concentration. In the Gromov--Milman framework, a L\'{e}vy group is approximated by compact subgroups whose Haar measures concentrate. Then for probability preserving Borel actions of such groups, one averages observables over these compact subgroups using their Haar measures, and Glasner--Tsirelson--Weiss' argument exploits Gromov--Milman's observation about degeneracy of measures. In the nonsingular setting, the averages are no longer Haar averages, but Haar averages tilted by Radon--Nikodym kernels. Thus one has to control families of tilted Haar measures, and ordinary concentration alone does not seem to provide such control.

\smallskip

This leads to Wasserstein stability. Given a compact metrizable group \(K\) with Haar measure \(m\) and compatible right-invariant metric \(d\), we consider measurable kernels of densities \(\Phi=\{\phi_{z}\}\) on \(K\), and look at the oscillation of tilted laws \(\phi_{z}m\) on average. The local quantity is the maximal total variation distance between \(h.\phi_{z}\) and \(\phi_{z}\) obtained for small \(h\in K\), while the global quantity is the Wasserstein distance between \(\phi_{z}m\) and \(m\). Wasserstein stability asserts that local total variation controls global Wasserstein distance from Haar measure.

\smallskip

In the presence of Wasserstein stability, the proof of \cref{mthm1} proceeds as follows. For a nonsingular Borel action of a Wasserstein group, we average the Radon--Nikodym cocycles over the compact subgroups in a Wasserstein skeleton. Wasserstein stability shows that the corresponding tilted averages are asymptotically indistinguishable from Haar averages. This forces the original quasi-invariant measure to be invariant. Applying the same argument to every equivalent measure, an elementary argument then shows that the measure is supported on the fixed points.

\smallskip

An important precursor to our approach is Pestov's notion of convergence to invariance \cite[\S5]{Pestov2010}. Pestov's approach is formulated using the Monge--Kantorovich distance, which is equivalent to the Wasserstein distance used here by Kantorovich duality. Pestov used this idea to obtain fixed-point consequences beyond the class of L\'{e}vy groups, for groups which
\begin{center}
\begin{minipage}{0.94\textwidth}
    \emph{retain a sufficient amount of concentration, witnessed by nets of probability measures that at the same time concentrate and converge to invariance} \cite[p.~386]{Pestov2010}.
\end{minipage}
\end{center}
The framework of Wasserstein groups can be viewed as a nonsingular elaboration of this idea, in that it incorporates local-to-global interplay of total variation and Wasserstein distance for measurable kernels.

\section{Wasserstein stability}

The modern formulation of the concentration of measure phenomenon, and in particular its instance in topological groups, can be found in \cite{gromov1983, milman1988heritage, glasner2005automorphism, Pestov2006, Pestov2010}; see also \cite[\S1-2]{talagrand1996new} and the references therein.

\smallskip

We start with some basic definitions and notations. A {\bf measure metric space} is a triplet \(\left(M,d,m\right)\), where \(\left(M,d\right)\) is a compact metric space, and \(m\) is a Borel probability measure on \(M\). For a set \(A\subseteq M\) and \(\epsilon>0\), we denote by \(B_{d}\left(A,\epsilon\right)\) the \(\epsilon\)-neighborhood of \(A\) in \(\left(M,d\right)\). A measure metric space \(\left(K,d,m\right)\) is called a {\bf measure metric group} if \(K\) is a compact (metrizable) group, \(d\) is a compatible right-invariant metric on \(K\), and \(m\) is the Haar probability measure on \(K\). In this case, we abbreviate by \(B_{d}\left(r\right)\) the \(d\)-ball of radius \(r\) around the identity of \(K\).

\subsection{L\'{e}vy families}

The {\bf concentration function} of a measure metric space \(\left(M,d,m\right)\) is defined by
\[\alpha_{\left(M,d,m\right)}\left(\epsilon\right):=1-\inf\left\{m\left(B_{d}\left(A,\epsilon\right)\right)\mid A\subseteq M\text{ is Borel and }m\left(A\right)\geq 1/2\right\},\quad\epsilon>0.\]

\begin{defn}
A sequence \(\left(M_{n},d_{n},m_{n}\right)_{n\geq1}\) of measure metric spaces is a {\bf L\'{e}vy family} if
\[\lim_{n\to\infty}\alpha_{\left(M_{n},d_{n},m_{n}\right)}\left(\epsilon\right)=0\text{ for every }\epsilon>0.\]
Equivalently, for every sequence \(\left(A_{n}\subseteq M_{n}\right)_{n\geq1}\) of Borel sets satisfying \(\inf_{n\geq 1}m_{n}\left(A_{n}\right)>0\), one has
\[\lim_{n\to\infty}m_{n}\left(B_{d_{n}}\left(A_{n},\epsilon\right)\right)=1\text{ for every }\epsilon>0.\footnote{For the definition of the L\'{e}vy family, it makes no difference if one replaces \(1/2\) with some other \(c\in\left(0,1/2\right]\); see~\cite[\S1.3]{Pestov2006}.}\]
Another equivalent form is that for every sequence \(\left(f_{n}:M_{n}\to\mathbb{R}\right)_{n\geq 1}\) of \(1\)-Lipschitz functions,
\[\lim_{n\to\infty}m_{n}\left(\left|f_{n}-M_{f_{n}}\right|\geq\epsilon\right)=0\text{ for every }\epsilon>0,\]
where in general \(M_{f}\) stands for any real number satisfying \(m\left(f\leq M_{f}\right)\geq1/2\) and \(m\left(f\geq M_{f}\right)\geq1/2\).
\end{defn}

The equivalence of the three formulations is standard; see \cite[\S2]{milman1988heritage}, \cite[Thm.~1.4.1]{Pestov2006}.

\subsection{Wasserstein families}

In the next we define the notion of Wasserstein family for a sequence of measure metric groups. We prepare to that by considering more objects on measure metric spaces.

\smallskip

Let \(\left(M,d,m\right)\) be a measure metric space. Consider the class of densities on \(\left(M,m\right)\),
\[\mathcal{D}\left(M,m\right):=\{\phi\in L^{1}\left(M,m\right):\phi\geq0,\,\mathbb{E}_{m}\left[\phi\right]=1\}.\]
For \(\phi\in\mathcal{D}\left(M,m\right)\), denote by \(\phi m\) the probability measure it induces on \(M\), namely \(\phi m\left(A\right)=\int_{A}\phi dm\). We occasionally do not distinguish \(\phi\) from \(\phi m\).

\smallskip

The simultaneous structure of \(M\) as a measure space and as a metric space suggests two natural distances on \(\mathcal{D}\left(M,m\right)\). The first is measure-theoretic: the {\bf total variation distance} is given by
\[\mathbf{V}_{m}\left(\phi,\phi'\right):=\left\|\phi m-\phi' m\right\|_{\operatorname{TV}}=\frac{1}{2}\mathbb{E}_{m}\left[\left|\phi-\phi'\right|\right].\]
The second is metric-geometric: the {\bf Wasserstein distance} (specifically, the \(1\)-Wasserstein distance) is given by
\[\mathbf{W}_{d}\left(\phi m,\phi' m\right):=\sup\nolimits_{\mathrm{Lip}_{d}\left(f\right)\leq1}\left|\mathbb{E}_{\phi m}\left[f\right]-\mathbb{E}_{\phi' m}\left[f\right]\right|,\]
where the supremum ranges over the \(1\)-Lipschitz functions on \(\left(M,d\right)\). The two distances satisfy
\[\mathbf{W}_{d}\left(\phi m,\phi' m\right)\leq\operatorname{diam}\left(M,d\right)\mathbf{V}_{m}\left(\phi,\phi'\right).\]
Indeed, after subtracting a constant from a \(1\)-Lipschitz function \(f\), we may assume \(\left\|f\right\|_{\infty}\leq\frac{1}{2}\operatorname{diam}\left(M,d\right)\). This readily gives the inequality.

\begin{rem}
This form of the Wasserstein distance arises from the Kantorovich duality; see~\cite[Thm.~5.10, Rem.~6.5]{villani2009optimal} (with \emph{Particular Case~5.16} therein), \cite[Thm.~4.13]{vanHandel2016}. Throughout this work, \emph{Wasserstein distance} refers to the \(p=1\) case of the general definition of order \(p\in\left[1,\infty\right)\) \cite[Def.~6.4]{villani2009optimal}. The \(1\)-Wasserstein distance, which we refer to here simply as the Wasserstein distance, is also known in the literature as the \emph{Kantorovich--Rubinstein distance} or \emph{Monge--Kantorovich distance} \cite[Rem.~6.5]{villani2009optimal}, \cite[\S5]{Pestov2010}.
\end{rem}

Let now \(\left(K,d,m\right)\) be a measure metric group. Then \(\mathcal{D}\left(K,m\right)\) is a \(K\)-space with the action
\[h.\phi\left(g\right)=\phi\left(h^{-1}g\right).\]
Consider a measurable kernel \(\Phi=\{\phi_{z}\in\mathcal{D}\left(K,m\right):z\in Z\}\) defined on some probability space \(\left(Z,\zeta\right)\). We measure the variation of \(\Phi\) using two methods. The first is local and uses total variation:
\[\mathbf{v}_{m}\left(\Phi,r\right):=\sup_{h\in B_{d}\left(r\right)}\int_{Z}\mathbf{V}_{m}\left(h.\phi_{z},\phi_{z}\right)d\zeta\left(z\right),\quad r>0;\]
the second is global and uses Wasserstein distance:
\[\mathbf{w}_{d}\left(\Phi\right):=\int_{Z}\mathbf{W}_{d}\left(\phi_{z}m,m\right)d\zeta\left(z\right).\]

\smallskip

The {\bf Wasserstein modulus} of a measure metric group \(\left(K,d,m\right)\) is defined by
\[\Omega_{\left(K,d,m\right)}\left(r,\epsilon\right):=\sup\left\{\mathbf{w}_{d}\left(\Phi\right):\mathbf{v}_{m}\left(\Phi,r\right)\leq\epsilon\right\},\quad r,\epsilon>0,\]
where the supremum ranges over all \(\mathcal{D}\left(K,m\right)\)-valued measurable kernels \(\Phi\).

\begin{defn}
A sequence \(\left(K_{n},d_{n},m_{n}\right)_{n\geq1}\) of measure metric groups is a {\bf Wasserstein family} if there is a positive sequence \(r_{n}\to0\) such that
\[\Omega_{\left(K_{n},d_{n},m_{n}\right)}\left(r_{n},\epsilon_{n}\right)\to0\text{ whenever }\epsilon_{n}\to0.\]
Equivalently, for every sequence of measurable kernels \(\left(\Phi_{n}\right)_{n\geq 1}\), each \(\Phi_{n}\) is \(\mathcal{D}\left(K_{n},m_{n}\right)\)-valued, one has
\[\mathbf{v}_{m_{n}}\left(\Phi_{n},r_{n}\right)\to0\quad\Longrightarrow\quad \mathbf{w}_{d_{n}}\left(\Phi_{n}\right)\to0.\]
\end{defn}

The equivalence of these two conditions can be proved routinely.

\subsection{Wasserstein families concentrate}

The Wasserstein modulus measures local-to-global stability, and more precisely, how local total variation oscillations affect global Wasserstein oscillations. Therefore, the Wasserstein modulus should not be treated as a concentrating function per se. Nevertheless, the following proposition gives a quantitative relation between the Wasserstein modulus and the concentration function, and yields that every Wasserstein family is a L\'{e}vy family.

\begin{prop}\label{prop:WassControlsLevy}
Let \(\left(K,d,m\right)\) be a measure metric group. Then for every \(0<\epsilon\leq1\) and every \(r>0\),
\[\epsilon^{2}\cdot\alpha_{\left(K,d,m\right)}\left(\epsilon\right)\leq10\cdot\Omega_{\left(K,d,m\right)}\left(r,r/2\right).\]
\end{prop}

The following corollary is rather immediate:

\begin{cor}\label{cor:WassisLevy}
Every Wasserstein family is a L\'{e}vy family.
\end{cor}

The proof of \cref{prop:WassControlsLevy} relies on the following lemma:

\begin{lem}\label{lem:setToDensity}
Let \(\left(K,d,m\right)\) be a measure metric group and \(A\subseteq K\) a measurable set with \(m\left(A\right)\geq 1/2\). Then for every \(0<\epsilon\leq1\) and every \(r>0\), there is a density \(\phi\in\mathcal{D}\left(K,m\right)\) such that
\[\mathbf{v}_{m}\left(\phi,r\right)\leq\frac{r}{2}\quad\text{and}\quad\mathbf{w}_{d}\left(\phi\right)\geq\frac{\epsilon^{2}m\left(K\setminus B_{d}\left(A,\epsilon\right)\right)}{10},\]
where \(\phi\) is viewed as a one point kernel, and so \(\mathbf{w}_{d}\left(\phi\right)=\mathbf{W}_{d}\left(\phi m,m\right)\).
\end{lem}

\begin{proof}[Proof of \cref{lem:setToDensity}]
Put \(B:=K\backslash B_{d}\left(A,\epsilon\right)\) and \(t:=\epsilon/2\). Define
\[f:K\to\left[0,t\right],\quad f\left(g\right):=\min\left\{d\left(g,A\right),t\right\}.\]
Then \(f\) is \(1\)-Lipschitz, \(f\mid_{A}\equiv0\), and \(f\mid_{B}\equiv t\). Define \(\phi\in\mathcal{D}\left(K,m\right)\) by
\[\phi\left(g\right):=\frac{1}{C}\cdot\left(1+f\left(g\right)\right),\quad C:=1+\mathbb{E}_{m}\left[f\right].\]
For \(h\in B_{d}\left(r\right)\), using right-invariance of \(d\), one has \(\left|f\left(h^{-1}g\right)-f\left(g\right)\right|\leq r\). Therefore,
\[\mathbf{V}_{m}\left(h.\phi,\phi\right)=\frac{1}{2C}\int_{K}\left|f\left(h^{-1}g\right)-f\left(g\right)\right|dm\left(g\right)\leq\frac{r}{2C}\leq\frac{r}{2}.\]
Thus \(\mathbf{v}_{m}\left(\phi,r\right)\leq r/2\). On the other hand, as \(f\) is \(1\)-Lipschitz,
\[\mathbf{W}_{d}\left(\phi m,m\right)\geq\int_{K}f\phi dm-\int_{K}fdm=\frac{1}{C}\operatorname{Var}_{m}\left(f\right),\]
where the last equality is a direct computation. We now make use of the following basic fact: for a random variable \(W\) and events \(E\subseteq\left\{W=a\right\}\) and \(F\subseteq\left\{W=b\right\}\), one has \(\mathrm{Var}\left(W\right)\geq\mathbb{P}\left(E\right)\mathbb{P}\left(F\right)\left(a-b\right)^{2}\). In our case, apply this to \(W=f\) with \(E=A\subseteq\left\{f=0\right\}\) and \(F=B\subseteq\left\{f=t\right\}\), and obtain
\[\operatorname{Var}_{m}\left(f\right)\geq\frac{m\left(B\right)t^{2}}{2}.\]
Since \(m\left(A\right)\geq 1/2\) we have \(\mathbb{E}_{m}\left[f\right]\leq t\cdot m\left(K\backslash A\right)\leq t/2=\epsilon/4\leq1/4\), so that \(C\leq 5/4\), and we obtain
\[\mathbf{W}_{d}\left(\phi m,m\right)\geq\frac{m\left(B\right)t^{2}}{2C}\geq\frac{2m\left(B\right)t^{2}}{5}=\frac{m\left(B\right)\epsilon^{2}}{10}.\qedhere\]
\end{proof}

\begin{proof}[Proof of \cref{prop:WassControlsLevy}]
Let \(A\subseteq K\) be measurable with \(m\left(A\right)\geq\frac{1}{2}\), and put \(B:=K\backslash B_{d}\left(A,\epsilon\right)\). Pick \(\phi\in\mathcal{D}\left(K,m\right)\) as in \cref{lem:setToDensity}. By the definition of \(\Omega_{\left(K,d,m\right)}\), this gives
\[\epsilon^{2}m\left(B\right)\leq10\cdot\Omega_{\left(K,d,m\right)}\left(r,r/2\right).\]
Taking the supremum over all \(A\subseteq K\) with \(m\left(A\right)\geq\frac{1}{2}\), we obtain the desired inequality.
\end{proof}

Let us compute concentration and Wasserstein stability for the family of discrete cubes with normalized Hamming metrics. This example will demonstrate that Wasserstein stability is strictly stronger than measure concentration. The concentration part of the following proposition is classical; see~\cite[Exm.~2.3, \S4.A.4]{milman1988heritage} as a discrete isoperimetric inequality, and \cite[p.~5, Exm.~4.3.7]{Pestov2006} as a geometric law of large numbers. For the Wasserstein stability part, one of the implications is proved using the later \cref{cor:qutifilt}.

\begin{prop}\label{prop:cubes}
For \(n\geq 1\), let \(K_{n}:=\left(\mathbb{Z}/2\mathbb{Z}\right)^{n}\) and let \(m_{n}\) be its Haar probability measure. Fix a parameter \(\beta>0\) and define the normalized Hamming metric
\[d_{n}\left(x,y\right):=\frac{\#\left\{1\leq j\leq n:x_{j}\neq y_{j}\right\}}{n^{\beta}}.\]
\begin{enumerate}
    \item \(\left(K_{n},d_{n},m_{n}\right)_{n\geq1}\) is a L\'{e}vy family if and only if \(\beta>1/2\). In fact, when \(\beta>1/2\) one has
    \[\alpha_{\left(K_{n},d_{n},m_{n}\right)}\left(\epsilon\right)\leq2\exp\Big(-\frac{\epsilon^{2}}{16}n^{2\beta-1}\Big),\quad\epsilon>0.\]
    \item \(\left(K_{n},d_{n},m_{n}\right)_{n\geq1}\) is a Wasserstein family if and only if \(\beta\geq 1\). In fact, when \(\beta\geq 1\) one has
    \[\Omega_{\left(K_{n},d_{n},m_{n}\right)}\left(2n^{-\beta},\epsilon\right)\leq2\epsilon,\quad\epsilon>0.\]
\end{enumerate}
In particular, if \(1/2<\beta<1\), then \(\left(K_{n},d_{n},m_{n}\right)_{n\geq 1}\) is a L\'{e}vy family but not a Wasserstein family.
\end{prop}

\begin{proof}
For (1), by \cite[Thms.~4.2.5, 4.3.19]{Pestov2006} one has
\[\alpha_{\left(K_{n},d_{n},m_{n}\right)}\left(\epsilon\right)\leq2\exp\Big(-\frac{\epsilon^{2}}{16}n^{2\beta-1}\Big).\]
This shows that \(\left(K_{n},d_{n},m_{n}\right)_{n\geq1}\) is a L\'{e}vy family when \(\beta>1/2\). Now if \(\beta\leq1/2\), consider the \(1\)-Lipschitz function \(F_{n}\left(x\right):=d_{n}\left(x,e_{n}\right)=\frac{1}{n^{\beta}}\sum_{j=1}^{n}x_{j}\), where \(e_{n}\in K_{n}\) is the identity. Since \(\sum_{j=1}^{n}x_{j}\sim\operatorname{Bin}\left(n,1/2\right)\),
\[m_{n}\Big(\big|F_{n}-\int_{K_{n}}F_{n}dm_{n}\big|>\epsilon\Big)=\mathbb{P}\Big(\frac{\left|\operatorname{Bin}\left(n,1/2\right)-n/2\right|}{\sqrt{n}}>\epsilon n^{\beta-1/2}\Big).\]
By the central limit theorem, this normalization converges in distribution to \(Z\sim N\left(0,1/4\right)\). If \(\beta=1/2\), this probability converges to \(\mathbb{P}\left(\left|Z\right|>\epsilon\right)>0\), and if \(0<\beta<1/2\), this probability converges to \(\mathbb{P}\left(\left|Z\right|>0\right)=1\). In either case it does not converge to \(0\), and hence \(\left(K_{n},d_{n},m_{n}\right)_{n\geq 1}\) is not a L\'{e}vy family.

\smallskip

For (2), assume first that \(\beta<1\). Define the density \(\phi_{n}\in\mathcal{D}\left(K_{n},m_{n}\right)\) by
\[\phi_{n}\left(x\right):=1-n^{\beta-1}+n^{\beta-1}2^{n}\mathbf{1}_{\left\{e_{n}\right\}}\left(x\right).\]
As \(m_{n}\left(\left\{e_{n}\right\}\right)=2^{-n}\), this gives the probability measure
\[\phi_{n}m_{n}=\left(1-n^{\beta-1}\right)m_{n}+n^{\beta-1}\delta_{e_{n}}.\]
Consider again the \(1\)-Lipschitz function \(F_{n}\left(x\right):=\frac{1}{n^{\beta}}\sum_{j=1}^{n}x_{j}\). Then
\[\mathbb{E}_{m_{n}}\left[F_{n}\right]=\frac{n}{2n^{\beta}}=\frac{1}{2n^{\beta-1}}\quad\text{and}\quad\mathbb{E}_{\phi_{n}m_{n}}\left[F_{n}\right]=\left(1-n^{\beta-1}\right)\frac{n}{2n^{\beta}}=\frac{1}{2n^{\beta-1}}-\frac{1}{2}.\]
We obtain
\begin{equation}\label{eq:tvexm}
\mathbf{w}_{d_{n}}\left(\phi_{n}\right)=\mathbf{W}_{d_{n}}\left(\phi_{n}m_{n},m_{n}\right)\geq\mathbb{E}_{m_{n}}\left[F_{n}\right]-\mathbb{E}_{\phi_{n}m_{n}}\left[F_{n}\right]=\frac{1}{2}.
\end{equation}
We now estimate the local total variation oscillation of \(\phi_{n}\). For every \(h\in K_{n}\),
\[\left(h.\phi_{n}\right)m_{n}=\left(1-n^{\beta-1}\right)m_{n}+n^{\beta-1}\delta_{h},\]
and therefore,
\[\mathbf{V}_{m_{n}}\left(h.\phi_{n},\phi_{n}\right)=n^{\beta-1}\cdot\mathbf{V}_{m_{n}}\left(\delta_{h},\delta_{e_{n}}\right)\leq n^{\beta-1}.\]
It follows that for whatever positive sequence \(r_{n}\) one may choose,
\begin{equation}\label{eq:Wassexm}
\mathbf{v}_{m_{n}}\left(\phi_{n},r_{n}\right)\leq n^{\beta-1}\to0.
\end{equation}
Then both \eqref{eq:tvexm} and \eqref{eq:Wassexm} together show that no positive sequence \(r_{n}\to0\) can witness Wasserstein stability, and \(\left(K_{n},d_{n},m_{n}\right)_{n\geq1}\) is not a Wasserstein family when \(\beta<1\).

\smallskip

Assume now that \(\beta\geq1\). For \(0\leq i\leq n\) let the subgroup
\[H_{n,i}:=\left\{x\in K_{n}:x_{i+1}=\dotsm=x_{n}=0\right\}\leq K_{n},\]
and so \(\left\{e_{K_{n}}\right\}=H_{n,0}\leq H_{n,1}\leq\dotsm\leq H_{n,n}=K_{n}\). The nontrivial coset in \(H_{n,i}/H_{n,i-1}\) has the representative \(\delta_{i}\), the vector supported on the \(i^{\mathrm{th}}\) coordinate, so \(d_{n}\left(e_{K_{n}},\delta_{i}\right)=n^{-\beta}\). Then this martingale chain has length bounded by \(\sum_{i=1}^{n}n^{-\beta}=n^{1-\beta}\leq1\) and scale bounded by \(n^{-\beta}\to0\) (see \cref{dfn:marchain}), so by \cref{prop:MSS},
\[\Omega_{\left(K_{n},d_{n},m_{n}\right)}\left(n^{-\beta},\epsilon\right)\leq2\epsilon,\quad\epsilon>0.\]
This readily implies that \(\left(K_{n},d_{n},m_{n}\right)_{n\geq1}\) is a Wasserstein family.
\end{proof}

\section{A functional inequality for Wasserstein stability via martingales}\label{sct:stabineq}

Using martingale steps to bound concentration is a classical theme since Azuma's inequality; see \cite[\S4.B.1]{milman1988heritage}, \cite[\S4]{talagrand1996new}, \cite[\S4.3]{Pestov2006}, \cite[\S3.2]{vanHandel2016} and the references therein. In the following, we present a functional inequality using martingales, in the spirit of Milman--Schechtman inequality \cite[I, Thm.~7.12]{milman1986asymptotic} (see also \cite[Thm.~4.5.3]{Pestov2006}), that is designed to control Wasserstein stability. Recently Schneider established such an inequality in terms of invariant means \cite[Thm.~1.1, Cor.~4.9]{schneider2023concentration}.

\smallskip

Those concentration inequalities, Azuma's inequality and Milman--Schechtman--Schneider inequalities, are governed by an \(\ell^{2}\)-sum of martingale increments. In contrast, our martingale argument for this functional inequality is governed by an \(\ell^{1}\)-sum of martingale increments, which reflects the stronger nature of Wasserstein stability over classical concentration.

\smallskip

For a measure metric group \(\left(K,d,m\right)\) and constants \(r,l>0\), consider the functional inequality
\begin{equation}\label{eq:funineq}\tag{\(\lozenge\)}
\mathbf{w}_{d}\left(\Phi\right)\leq l\mathbf{v}_{m}\left(\Phi,r\right),
\end{equation}
for all \(\mathcal{D}\left(K,m\right)\)-valued measurable kernel \(\Phi\). Equivalently, the Wasserstein modulus satisfies
\[\Omega_{\left(K,d,m\right)}\left(r,\epsilon\right)\leq l\epsilon,\quad\epsilon>0.\]

\smallskip

We now present the basics we need to establish the functional inequality \eqref{eq:funineq}.

\begin{defn}\label{dfn:marchain}
Let \(\left(K,d,m\right)\) be a measure metric group. A {\bf martingale chain} is a sequence
\[\left\{e_{K}\right\}=H_{0}\leq H_{1}\leq\dotsm\leq H_{N}=K\]
of compact subgroups. For such a martingale chain, define the quotient radii
\[\rho_{i}:=\sup_{g\in H_{i}}d\left(e_{K},gH_{i-1}\right)=\sup_{g\in H_{i}}\inf_{h\in H_{i-1}}d\left(e_{K},gh\right),\quad 1\leq i\leq N,\]
and define the {\bf length} and the {\bf scale} of the martingale chain by
\[\ell:=\sum\nolimits_{i=1}^{N}\rho_{i}\quad\text{and}\quad\varrho:=\max_{1\leq i\leq N}\rho_{i}.\]
\end{defn}

\begin{rem}
For a martingale chain as above, Schneider associates the quantities
\[\delta_{i}:=\sup\nolimits_{g\in K}\operatorname{diam}\big(H_{i}/H_{i-1},d^{g}_{K/H_{i-1}}\big),\quad d^{g}_{K/H_{i-1}}\left(xH_{i-1},yH_{i-1}\right):=\inf\nolimits_{h\in H_{i-1}}d\left(gx,gyh\right).\]
Our radii \(\rho_{i}\) are measured in the same quotient geometry, and satisfy
\[\rho_{i}\leq\operatorname{diam}\big(H_{i}/H_{i-1},d^{e_{K}}_{K/H_{i-1}}\big)\leq\delta_{i}.\]
Note that when \(d\) is bi-invariant, the dependence on \(g\) disappears and \(\rho_{i}=\delta_{i}\). Similarly to Azuma's inequality, the Milman--Schechtman--Schneider inequalities are governed by the \(\ell^{2}\)-sum \(\left(\sum_{i}\delta_{i}^{2}\right)^{1/2}\), while our inequality is governed by the \(\ell^{1}\)-sum \(\sum_{i}\rho_{i}\). See \cite[I, Thm.~7.12(i)]{milman1986asymptotic}, \cite[Def.~4.6, Thm.~4.8, Cor.~4.9]{schneider2023concentration}.
\end{rem}

\begin{prop}\label{prop:MSS}
Let \(\left(K,d,m\right)\) be a measure metric group admitting a martingale chain with length \(\ell\) and scale \(\varrho\). Then \(\left(K,d,m\right)\) satisfies the functional inequality \eqref{eq:funineq} with constants \(r=\varrho,\,l=2\ell\), namely
\[\Omega_{\left(K,d,m\right)}\left(\varrho,\epsilon\right)\leq2\ell\epsilon,\quad\epsilon>0.\]
\end{prop}

We first define some basic objects. Given a martingale chain \(\left\{e_{K}\right\}=H_{0}\leq H_{1}\leq\dotsm\leq H_{N}=K\), for \(1\leq i\leq N\) let \(\mathcal{I}_{i}\) be the sub-\(\sigma\)-algebra of \(\mathcal{B}\left(K\right)\) consisting of left \(H_{i}\)-invariant sets, and put the associated conditional expectation
\[E_{i}:=\mathbb{E}_{m}\left[\cdot\mid\mathcal{I}_{i}\right].\]
The association \(q\to q\cap\overline{B}_{d}\left(e_{K},\rho_{i}\right)\) defines a Borel map from the compact space \(H_{i}/H_{i-1}\) to the space of compact subsets of \(H_{i}\). Then by the Kuratowski--Ryll-Nardzewski selection theorem \cite[Thm.~(12.13)]{kechris2012descriptive}, there is a Borel section
\[\sigma_{i}:H_{i}/H_{i-1}\to H_{i}\]
such that
\[\sigma_{i}\left(q\right)\in q\quad\text{and}\quad d\left(e_{K},\sigma_{i}\left(q\right)\right)\leq\rho_{i},\quad q\in H_{i}/H_{i-1}.\]
Let \(\bar{m}_{i}\) be the quotient Haar probability measure on the left coset space \(H_{i}/H_{i-1}\). By Weil's integration formula \cite[Thm.~(2.56)]{folland1995course}, for every \(f\in L^{1}\left(K,m\right)\) we have the following version of \(E_{i}f\):
\begin{equation}\label{eq:Weil}
E_{i}f\left(g\right)=\int_{H_{i}}f\left(hg\right)dm_{H_{i}}\left(h\right)=\int_{H_{i}/H_{i-1}}\int_{H_{i-1}}f\left(\sigma_{i}\left(q\right)hg\right)dm_{H_{i-1}}\left(h\right)d\bar{m}_{i}\left(q\right)\quad m\text{-a.e. }g\in K.
\end{equation}

\begin{lem}\label{lem:martcont}
In the setup above, for every \(1\leq i\leq N\) the following hold.
\begin{itemize}
    \item For every \(1\)-Lipschitz function \(f:K\to\mathbb{R}\),
    \[\left\|E_{i}f-E_{i-1}f\right\|_{L^{\infty}\left(K\right)}\leq\rho_{i}.\]
    \item For every \(\phi\in L^{1}\left(K,m\right)\),
    \[\left\|E_{i}\phi-E_{i-1}\phi\right\|_{L^{1}\left(m\right)}\leq\int_{H_{i}/H_{i-1}}2\mathbf{V}_{m}\big(\sigma_{i}\left(q\right)^{-1}.\phi,\phi\big)d\bar{m}_{i}\left(q\right).\]
\end{itemize}
Consequently, for every density \(\phi\in\mathcal{D}\left(K,m\right)\),
\[\mathbf{W}_{d}\left(\phi m,m\right)\leq\sum\nolimits_{i=1}^{N}2\rho_{i}\int_{H_{i}/H_{i-1}}\mathbf{V}_{m}\big(\sigma_{i}\left(q\right)^{-1}.\phi,\phi\big)d\bar{m}_{i}\left(q\right).\]
\end{lem}

\begin{proof}[Proof of \cref{lem:martcont}]
Fix \(1\leq i\leq N\). Using Weil's integration formula \eqref{eq:Weil} and the right-invariance of \(d\), for every \(1\)-Lipschitz function \(f:K\to\mathbb{R}\) and \(m\)-a.e. \(g\in K\) we have
\begin{align*}
\left|E_{i}f\left(g\right)-E_{i-1}f\left(g\right)\right|
&\leq\int_{H_{i}/H_{i-1}}\int_{H_{i-1}}\left|f\left(\sigma_{i}\left(q\right)hg\right)-f\left(hg\right)\right|dm_{H_{i-1}}\left(h\right)d\bar{m}_{i}\left(q\right)\\
&\leq\int_{H_{i}/H_{i-1}}\int_{H_{i-1}}d\left(\sigma_{i}\left(q\right)hg,hg\right)dm_{H_{i-1}}\left(h\right)d\bar{m}_{i}\left(q\right)\\
&=\int_{H_{i}/H_{i-1}}d\left(e_{K},\sigma_{i}\left(q\right)\right)d\bar{m}_{i}\left(q\right)\\
&\leq\rho_{i}.
\end{align*}
This shows that \(\left\|E_{i}f-E_{i-1}f\right\|_{L^{\infty}\left(K\right)}\leq\rho_{i}\). Similarly, for every \(\phi\in L^{1}\left(K,m\right)\) we have
\begin{align*}
\left\|E_{i}\phi-E_{i-1}\phi\right\|_{L^{1}\left(m\right)}
&\leq\int_{H_{i}/H_{i-1}}\int_{H_{i-1}}\int_{K}\left|\phi\left(\sigma_{i}\left(q\right)hg\right)-\phi\left(hg\right)\right|dm\left(g\right)dm_{H_{i-1}}\left(h\right)d\bar{m}_{i}\left(q\right)\\
&=\int_{H_{i}/H_{i-1}}\big\|\sigma_{i}\left(q\right)^{-1}.\phi-\phi\big\|_{L^{1}\left(m\right)}d\bar{m}_{i}\left(q\right)=\int_{H_{i}/H_{i-1}}2\mathbf{V}_{m}\big(\sigma_{i}\left(q\right)^{-1}.\phi,\phi\big)d\bar{m}_{i}\left(q\right).
\end{align*}

Now let \(\phi\in\mathcal{D}\left(K,m\right)\), and fix an arbitrary \(1\)-Lipschitz function \(f:K\to\mathbb{R}\). Put
\[\Delta_{i}f:=E_{i-1}f-E_{i}f\quad\text{and}\quad\Delta_{i}\phi:=E_{i-1}\phi-E_{i}\phi.\]
Using \(E_{0}f=f\), \(E_{N}f=\int_{K}fdm\), \(E_{0}\phi=\phi\), and \(E_{N}\phi=1\), we have the telescoping identities
\[f-\int_{K}fdm=E_{0}f-E_{N}f=\sum\nolimits_{i=1}^{N}\Delta_{i}f,\quad\text{and}\quad\phi-1=E_{0}\phi-E_{N}\phi=\sum\nolimits_{i=1}^{N}\Delta_{i}\phi.\]
We now claim to have the martingale telescoping identity
\begin{equation}
\begin{aligned}
\int_{K}f\cdot\left(\phi-1\right)dm
&=\int_{K}\left(f-E_{N}f\right)\cdot\left(\phi-E_{N}\phi\right)dm\\
&=\sum\nolimits_{i,j=1}^{N}\int_{K}\Delta_{i}f\cdot\Delta_{j}\phi dm=\sum\nolimits_{i=1}^{N}\int_{K}\Delta_{i}f\cdot\Delta_{i}\phi dm.
\end{aligned}\label{eq:MTID}
\end{equation}
Let us justify this computation. The first equality follows because \(\int_{K}\left(\phi-1\right)dm=0\). The second equality follows by the above telescoping identities. For the third equality, we need the standard fact
\[\int_{K}\Delta_{i}f\cdot\Delta_{j}\phi dm=0\,\text{ for all }i\neq j;\]
let us consider the case \(i<j\) (the case \(j<i\) is analogous): by definition \(\Delta_{j}\phi\) is \(\mathcal{I}_{j-1}\)-measurable, and since \(i<j\) we have \(\mathcal{I}_{j-1}\subseteq\mathcal{I}_{i}\), and so \(\Delta_{j}\phi\) is also \(\mathcal{I}_{i}\)-measurable. Then by the law of total expectation
\[E_{i}\Delta_{i}f=E_{i}\left[E_{i-1}f-E_{i}f\right]=E_{i}f-E_{i}f=0,\]
and hence, using again the law of total expectation, we deduce that
\[\int_{K}\Delta_{i}f\cdot\Delta_{j}\phi dm=\int_{K}E_{i}\left[\Delta_{i}f\cdot\Delta_{j}\phi\right]dm=\int_{K}E_{i}\Delta_{i}f\cdot\Delta_{j}\phi dm=0.\]
This justifies the third equality, and so \eqref{eq:MTID} is established. We can now complete the proof:
\begin{align*}
\left|\mathbb{E}_{\phi m}\left[f\right]-\mathbb{E}_{m}\left[f\right]\right|	
&=\Big|\int_{K}f\cdot\left(\phi-1\right)dm\Big|\leq\sum\nolimits_{i=1}^{N}\Big|\int_{K}\Delta_{i}f\cdot\Delta_{i}\phi dm\Big|\\
&\leq\sum\nolimits_{i=1}^{N}\left\Vert\Delta_{i}f\right\Vert_{L^{\infty}\left(K\right)}\cdot\left\Vert\Delta_{i}\phi\right\Vert_{L^{1}\left(m\right)}\leq\sum\nolimits_{i=1}^{N}2\rho_{i}\int_{H_{i}/H_{i-1}}\mathbf{V}_{m}\big(\sigma_{i}\left(q\right)^{-1}.\phi,\phi\big)d\bar{m}_{i}\left(q\right),
\end{align*}
where the first inequality is by \eqref{eq:MTID}, and the third inequality is by the first two parts of the lemma. Finally, since \(f\) is arbitrary and the right hand-side is independent of \(f\), the desired inequality follows.
\end{proof}

\begin{proof}[Proof of \cref{prop:MSS}]
Let \(\Phi=\{\phi_{z}:z\in Z\}\) be a \(\mathcal{D}\left(K,m\right)\)-valued measurable kernel. By \cref{lem:martcont} applied to each \(\phi_{z}\) and by Fubini's theorem,
\[\mathbf{w}_{d}\left(\Phi\right)\leq\sum\nolimits_{i=1}^{N}2\rho_{i}\int_{H_{i}/H_{i-1}}\int_{Z}\mathbf{V}_{m}\big(\sigma_{i}\left(q\right)^{-1}.\phi_{z},\phi_{z}\big)d\zeta\left(z\right)d\bar{m}_{i}\left(q\right).\]
For every \(1\leq i\leq N\) and \(q\in H_{i}/H_{i-1}\), we have by construction \(\sigma_{i}\left(q\right)^{-1}\in B_{d}\left(\varrho\right)\). Therefore,
\[\int_{Z}\mathbf{V}_{m}\big(\sigma_{i}\left(q\right)^{-1}.\phi_{z},\phi_{z}\big)d\zeta\left(z\right)\leq\mathbf{v}_{m}\left(\Phi,\varrho\right).\]
Both estimates give
\[\mathbf{w}_{d}\left(\Phi\right)\leq\sum\nolimits_{i=1}^{N}2\rho_{i}\mathbf{v}_{m}\left(\Phi,\varrho\right)=2\ell\mathbf{v}_{m}\left(\Phi,\varrho\right).\]
This establishes the functional inequality \eqref{eq:funineq} with \(r=\varrho,\,l=2\ell\).
\end{proof}

As we shall explain soon, after establishing the notion of Wasserstein group, the following corollary implies \cref{mthm2}:

\begin{cor}\label{cor:qutifilt}
Let \(\left(K_{n},d_{n},m_{n}\right)_{n\geq1}\) be a sequence of measure metric groups. Suppose that each \(\left(K_{n},d_{n},m_{n}\right)\) admits a martingale chain with length \(\ell_{n}\) and scale \(\varrho_{n}\), such that
\[\sup\nolimits_{n\geq1}\ell_{n}<+\infty\quad\text{and}\quad \varrho_{n}\to0.\]
Then \(\left(K_{n},d_{n},m_{n}\right)_{n\geq1}\) is a Wasserstein family.
\end{cor}

\begin{proof}[Proof of \cref{cor:qutifilt}]
Let \(L:=\sup_{n\geq1}\ell_{n}\). By \cref{prop:MSS}, for every \(\epsilon>0\),
\[\Omega_{\left(K_{n},d_{n},m_{n}\right)}\left(\varrho_{n},\epsilon\right)\leq2L\epsilon.\]
Thus, \(\Omega_{\left(K_{n},d_{n},m_{n}\right)}\left(\varrho_{n},\epsilon_{n}\right)\to0\) whenever \(\epsilon_{n}\to0\).
\end{proof}

\begin{exm}\label{exm:findimunit}
We illustrate Wasserstein stability for finite dimensional matrix groups. Let \(K_{n}=U\left(n\right)\) be the unitary group of an \(n\)-dimensional Hilbert space, let \(m_{n}\) be its Haar measure, and let the metric
\[d_{n}\left(u,v\right):=\operatorname{tr}_{n}\left(\left|u-v\right|\right),\]
where \(\operatorname{tr}_{n}:=\frac{1}{n}\operatorname{Tr}\) is the normalized trace on \(M_{n}\left(\mathbb{C}\right)\), and \(\left|a\right|:=\left(a^{*}a\right)^{1/2}\). Consider the subgroup chain
\[\left\{\mathrm{I}_{n}\right\}=U_{n}\left(0\right)\leq U_{n}\left(1\right)\leq\dotsm\leq U_{n}\left(n\right)=U\left(n\right),\]
where \(U_{n}\left(i\right)\cong U\left(i\right)\) is embedded in the upper-left corner. Every coset \(uU_{n}\left(i-1\right)\in U_{n}\left(i\right)/U_{n}\left(i-1\right)\) is determined by the unit vector \(u e_{i}\in\operatorname{span}\left\{e_{1},\dotsc,e_{i}\right\}\). Choose a unitary representative \(s\) in this coset which sends \(e_{i}\) to \(u e_{i}\) and is the identity on the orthogonal complement of \(\operatorname{span}\left\{e_{i},u e_{i}\right\}\). Then \(s-\mathrm{I}_{n}\) has rank at most \(2\), and therefore
\[\left\Vert\mathrm{I}_{n}-s\right\Vert_{1}\leq\mathrm{rank}\left(\mathrm{I}_{n}-s\right)\cdot\left\Vert\mathrm{I}_{n}-s\right\Vert_{\text{op}}\leq2\cdot\left(\left\Vert\mathrm{I}_{n}\right\Vert_{\text{op}}+\left\Vert s\right\Vert_{\text{op}}\right)\leq4,\]
where \(\left\|\cdot\right\|_{1}\) denotes the trace norm. It follows that
\[d_{n}\left(\mathrm{I}_{n},s\right)=\frac{1}{n}\left\|s-\mathrm{I}_{n}\right\|_{1}\leq\frac{4}{n}.\]
Thus the chain has scale at most \(4/n\) and length at most \(4\). Therefore, by \cref{prop:MSS} we obtain
\[\Omega_{\left(U\left(n\right),d_{n},m_{n}\right)}\left(4/n,\epsilon\right)\leq8\epsilon,\quad\epsilon>0.\]
Then in accordance with \cref{cor:qutifilt}, we get that \(\left(U\left(n\right),d_{n},m_{n}\right)_{n\geq 1}\) forms a Wasserstein family.
\end{exm}

\begin{rem}
This argument can be used to show that the unitary group of the hyperfinite \(\mathrm{II}_{1}\)-factor is a \emph{Wasserstein group} (see \cref{sct:unithyperf}). However, it does not make \(U\left(\mathcal{H}\right)\), the unitary group of an infinite dimensional Hilbert space, a \emph{Wasserstein group}. The reason is that the normalized trace-norm metrics
\[d_{n}\left(u,v\right)=\operatorname{tr}_{n}\left(\left|u-v\right|\right)\]
on \(U\left(n\right)\) are not induced from one compatible metric on \(U\left(\mathcal{H}\right)\). This is analogous to that \(\left(S_{n},d_{n},m_{n}\right)_{n\geq 1}\) is a L\'{e}vy family, where \(S_{n}\) is the finite symmetric group on \(n\) elements with the Hamming metric
\[d_{n}\left(\sigma,\tau\right):=\frac{1}{n}\left|\left\{1\leq i\leq n:\sigma\left(i\right)\neq\tau\left(i\right)\right\}\right|,\]
and still it does not make the Polish group \(S_{\infty}\) of all permutations of an infinite set, with the topology of pointwise convergence, into a L\'{e}vy group; see the discussion in \cite[p.~91]{Pestov2006}. In fact, the non-Archimedean Polish group \(S_{\infty}\) is far from being a L\'{e}vy group in a few aspects \cite{kechris2014dynamics,glasner2005spatial}.
\end{rem}

\section{Wasserstein and locally Wasserstein Polish groups}\label{sct:Wass}

Let \(G\) be a Polish group equipped with a compatible right-invariant metric \(d\). A {\bf skeleton} for \(G\) is a sequence of compact subgroups
\[K_{1}\leq K_{2}\leq\dotsm<G,\]
such that \(\mathcal{K}:=\bigcup_{n\geq1}K_{n}\) is dense in \(G\). For such a skeleton we have the sequence of measure metric groups \(\left(K_{n},d\mid_{K_{n}},m_{n}\right)_{n\geq 1}\), where \(m_{n}\) always stands for the Haar probability measure of \(K_{n}\).

\begin{defn}
A Polish group \(G\) is a {\bf L\'{e}vy group} if it admits a compatible right-invariant metric \(d\) and a skeleton \(\mathcal{K}=\bigcup_{n\geq1}K_{n}\) such that \(\left(K_{n},d\mid_{K_{n}},m_{n}\right)_{n\geq1}\) is a L\'{e}vy family.
\end{defn}

\begin{defn}\label{def:WassersteinGroup}
A Polish group \(G\) is a {\bf Wasserstein group} if it admits a compatible right-invariant metric \(d\) and a skeleton \(\mathcal{K}=\bigcup_{n\geq1}K_{n}\) such that \(\left(K_{n},d\mid_{K_{n}},m_{n}\right)_{n\geq1}\) is a Wasserstein family.
\end{defn}

By \cref{cor:WassisLevy}, every Wasserstein group is a L\'{e}vy group. For the converse, while there are L\'{e}vy families which are not Wasserstein families, as suggested in \cref{prop:cubes}, we do not know whether there is a L\'{e}vy group which is not a Wasserstein group. In \cref{app:Wass} we list many examples of well-known L\'{e}vy groups and prove that they are in fact Wasserstein groups. Now with the notion of Wasserstein group defined, we easily get:

\begin{proof}[Proof of \cref{mthm2}]
The proof follows from \cref{cor:qutifilt}, specialized to Wasserstein families which are formed by increasing sequence of compact subgroups of an ambient Polish group, and their metrics are all restrictions of one right-invariant compatible metric on the ambient group.
\end{proof}

We now pass to groups which are topologically generated by images of Wasserstein groups.

\begin{defn}\label{def:appWass}
A Polish group \(G\) is called {\bf locally Wasserstein} if there are Wasserstein groups \(H_{1},H_{2},\dotsc\) and continuous homomorphisms \(\iota_{n}:H_{n}\to G\), such that \(\bigcup\nolimits_{n\geq1}\iota_{n}\left(H_{n}\right)\) generates a dense subgroup of \(G\). Equivalently, for every nonempty open set \(O\subset G\), there are Wasserstein groups \(H_{1},\dotsc,H_{k}\) and continuous homomorphisms \(\iota_{j}:H_{j}\to G\), such that \(\iota_{1}\left(H_{1}\right)\dotsm\iota_{k}\left(H_{k}\right)\cap O\neq\emptyset\).
\end{defn}

Note that the second definition does not require an a priori countable collection of Wasserstein groups that does the work, but by second countability one can always move to a countable sub-collection that does the work.

\smallskip

Every Wasserstein group is locally Wasserstein, by taking \(H_{1}=G\) and \(\iota_{1}=\operatorname{id}_{G}\). In \cref{app:locWass} we give a few examples of L\'{e}vy groups which are locally Wasserstein groups.

\section{Nonsingular actions: generalities and the Borel lifting problem}\label{sct:nonsingspat}

We collect the standard facts and necessary background for the Borel lifting problem in nonsingular actions needed for the proof of \cref{mthm1}. For more background on near actions and spatial models see \cite[Introduction]{glasner2005automorphism} and \cite[\S7.1]{Pestov2006}.

\subsection{Borel \(G\)-spaces}

Let \(G\) be a Polish group. A {\bf Borel \(G\)-space}, denoted \(G\curvearrowright X\), is a standard Borel space \(X\) equipped with a (jointly) Borel action of \(G\), that is a measurable map
\[G\times X\to X,\quad\left(g,x\right)\mapsto g.x,\]
satisfying the axioms of group actions. When \(X\) is a compact metric space and the action is (jointly) continuous, we call it a {\bf compact \(G\)-space}. The following celebrated theorem of Becker--Kechris \cite[Thm.~2.6.6]{becker1996descriptive} generalizes a classical theorem of Mackey--Varadarajan from locally compact Polish groups to general Polish groups. It plays an important role in Glasner--Tsirelson--Weiss' proof, as well as in ours.

\begin{thm}[Becker--Kechris]\label{thm:BecKec}
Let \(G\) be a Polish group. Then every Borel \(G\)-space admits a \(G\)-equivariant Borel embedding into a compact \(G\)-space.
\end{thm}

For a Borel \(G\)-space \(X\), write
\[\operatorname{Fix}\left(G\curvearrowright X\right):=\left\{x\in X:g.x=x\,\,\forall g\in G\right\}.\]
When \(X\) is a compact \(G\)-space this set is closed. Then from \cref{thm:BecKec} it follows that \(\operatorname{Fix}\left(G\curvearrowright X\right)\) is a Borel set for every Borel \(G\)-space \(X\).

\subsection{Measure preserving and nonsingular actions}

A {\bf standard measure space} \(\left(X,\mu\right)\) is a standard Borel space equipped with a Borel \(\sigma\)-finite measure. When \(G\curvearrowright X\) is a Borel \(G\)-space and \(\mu\) is a Borel \(\sigma\)-finite measure on it, we will call the standard measure space \(\left(X,\mu\right)\) {\bf measure preserving} if \(g_{\ast}\mu=\mu\) for every \(g\in G\), and {\bf nonsingular} if \(g_{\ast}\mu\sim\mu\) for every \(g\in G\). Here \(\sim\) stands for mutual absolute continuity of measures. We may also say that \(\mu\) is {\bf invariant} or {\bf quasi-invariant} measure for \(G\curvearrowright X\). We will denote this system by \(G\curvearrowright\left(X,\mu\right)\) and call it a {\bf probability preserving Borel action} or {\bf nonsingular Borel action}.\footnote{In the terminology of \cite{glasner2005automorphism}, these are called \emph{spatial actions}.} Since every \(\sigma\)-finite measure is equivalent to a probability measure, and since nonsingularity depends only on the measure class, in the nonsingular setup we may always replace \(\mu\) by an equivalent probability measure.

\smallskip

For a standard probability space \(\left(X,\mu\right)\), let \(\mathrm{Aut}\left(X,\mu\right)\) be the group of probability preserving automorphism, namely equivalence classes modulo equality \(\mu\)-a.e. of measure preserving Borel transformations, and let \(\mathrm{Aut}^{\ast}\left(X,\mu\right)\) be the analogous group of nonsingular automorphisms. The groups \(\mathrm{Aut}\left(X,\mu\right)\) and \(\mathrm{Aut}^{\ast}\left(X,\mu\right)\) are Polish groups with their usual weak topologies; for the measure preserving case see \cite[\S4.5]{Pestov2006}, and for the nonsingular case see \cite[\S1]{danilenko1995topological} and \cite[\S4.1]{giordano2007some}, where the weak topology is described equivalently through the nonsingular Koopman representation. Note that the isomorphism type of \(\mathrm{Aut}^{\ast}\left(X,\mu\right)\) depends only on the measure class of \(\mu\), and so to describe the topology one may assume that \(\mu\) is a probability measure. In this case it was shown by Le Ma\^{\i}tre that the Polish topology on \(\mathrm{Aut}^{\ast}\left(X,\mu\right)\) is unique \cite[Thm.~1.4]{lema2022polish}.

\smallskip

A {\bf measure preserving near action} or {\bf nonsingular near action} is a continuous homomorphism
\[G\to\mathrm{Aut}\left(X,\mu\right),\quad\text{respectively}\quad G\to\mathrm{Aut}^{\ast}\left(X,\mu\right).\]
Every nonsingular Borel action \(G\curvearrowright\left(X,\mu\right)\) induces a nonsingular near action by sending \(g\in G\) to the equivalence class of the transformation \(x\mapsto g.x\) in \(\mathrm{Aut}^{\ast}\left(X,\mu\right)\). Indeed, the resulting homomorphism \(G\to\mathrm{Aut}^{\ast}\left(X,\mu\right)\) is measurable and hence continuous by Pettis' automatic continuity property \cite[\S9.C]{kechris2012descriptive}. A {\bf spatial model} of a nonsingular near action \(\tau:G\to\mathrm{Aut}^{\ast}\left(X,\mu\right)\), is a nonsingular Borel action \(G\curvearrowright\left(X,\mu\right)\) whose induced near action is \(\tau\). We can now put the Borel lifting problem in a precise way:

\begin{BLP}\phantomsection\label{blp:nonsingular}
Given a nonsingular near action, does it admit a spatial model?\\
That is, for a nonsingular near action \(\tau:G\to\mathrm{Aut}^{\ast}\left(X,\mu\right)\), does there exist a Borel action \(G\curvearrowright X\) such that \(\mu\) is quasi-invariant and, for each \(g\in G\), the transformation \(x\mapsto g.x\) represents \(\tau\left(g\right)\in\mathrm{Aut}^{\ast}\left(X,\mu\right)\)?
\end{BLP}

\subsection{Radon--Nikodym cocycles and nonsingular Koopman representations}

Let \(G\curvearrowright\left(X,\mu\right)\) be a nonsingular Borel action. With every \(g\in G\) is associated the Radon--Nikodym derivative
\[\frac{dg^{-1}\mu}{d\mu}:X\to\mathbb{R}_{>0},\]
defined \(\mu\)-a.e. by the identity
\[\int_{X}f\left(g.x\right)\frac{dg^{-1}\mu}{d\mu}\left(x\right)d\mu\left(x\right)=\int_{X}f\left(x\right)d\mu\left(x\right),\quad f\in L^{1}\left(\mu\right).\]
The Radon--Nikodym derivatives form an almost cocycle:
\[\frac{d\left(gh\right)^{-1}\mu}{d\mu}\left(x\right)=\frac{dg^{-1}\mu}{d\mu}\left(h.x\right)\cdot\frac{dh^{-1}\mu}{d\mu}\left(x\right)\quad\text{for }\mu\text{-a.e. }x\in X,\]
where the exceptional null set may depend on \(g,h\). A pointwise-defined multiplicative {\bf cocycle} for a Borel \(G\)-space \(X\) is a Borel map
\[\nabla:G\times X\to\mathbb{R}_{>0},\quad\left(g,x\right)\mapsto\nabla_{g}\left(x\right),\]
such that
\[\nabla_{gh}\left(x\right)=\nabla_{g}\left(h.x\right)\nabla_{h}\left(x\right),\quad g,h\in G,\,x\in X.\]
A cocycle \(\nabla\) for a compact subgroup \(K<G\) is called a {\bf pointwise version} of the Radon--Nikodym almost cocycle of the nonsingular \(K\)-subaction if, for every \(g\in K\),
\[\nabla_{g}\left(x\right)=\frac{dg^{-1}\mu}{d\mu}\left(x\right)\quad\text{for }\mu\text{-a.e. }x\in X.\]
In general, a Radon--Nikodym almost cocycle need not admit a pointwise version; see \cite{becker2013cocycles}. However, for locally compact Polish groups, by the Mackey cocycle theorem (see \cite[Lem.~5.26, p.~179]{varadarajan1968geometry}) a pointwise version can always be found. In particular, whenever \(K<G\) is compact, we may choose a cocycle
\[\nabla^{K}:K\times X\to\mathbb{R}_{>0}\]
which is a pointwise version of the Radon--Nikodym almost cocycle of the nonsingular \(K\)-subaction.

\smallskip

The nonsingular Koopman representation of a nonsingular Borel action \(G\curvearrowright\left(X,\mu\right)\) is given by
\[\kappa:G\to U\left(L^{2}\left(\mu\right)\right),\quad\kappa_{g}f\left(x\right):=\Big(\tfrac{dg^{-1}\mu}{d\mu}\left(x\right)\Big)^{1/2}\cdot f\left(g.x\right).\]
The same formula applies to nonsingular near actions, modulo null sets. The following is a classical fact:

\begin{fct}\label{fct:koopmancont}
Let \(G\) be a Polish group and \(G\to\mathrm{Aut}^{\ast}\left(X,\mu\right)\) a nonsingular near action. Then the associated nonsingular Koopman representation \(\kappa\) is a strongly continuous representation. In particular, when \(\mu\) is a probability measure so that \(\mathbf{1}\in L^{2}\left(\mu\right)\),
\[\big\|\big(\tfrac{dg^{-1}\mu}{d\mu}\big)^{1/2}-1\big\|_{L^{2}\left(\mu\right)}\to0\,\,\text{ as }\,\,g\to e_{G}.\]
\end{fct}

This fact can be seen in a few ways. One way is to verify the measurability of the nonsingular Koopman representation, as a homomorphism between Polish groups \(G\to U\left(L^{2}\left(\mu\right)\right)\), and then apply Pettis' automatic continuity property \cite[\S9.C]{kechris2012descriptive}. Another way is to use that the weak topology on \(\mathrm{Aut}^{\ast}\left(X,\mu\right)\) can be defined as the pullback of the strong operator topology of \(U\left(L^{2}\left(\mu\right)\right)\) along the 'universal' nonsingular Koopman map, namely \(\mathrm{Aut}^{\ast}\left(X,\mu\right)\to U\left(L^{2}\left(\mu\right)\right)\), \(T\mapsto\big(f\mapsto\big(\tfrac{dT^{-1}\mu}{d\mu}\big)^{1/2}\cdot f\circ T\big)\); see \cite[\S1]{danilenko1995topological}, \cite[\S4.1]{giordano2007some}, \cite[Ex.~(17.46)]{kechris2012descriptive}. The nonsingular Koopman representation of \(G\curvearrowright\left(X,\mu\right)\) is the composition of this map with the continuous map \(G\to\mathrm{Aut}^{\ast}\left(X,\mu\right)\), hence it is strongly continuous.

\section{Proof of the main theorem}\label{sct:proof-main-theorem}

\subsection{Technical lemmas}

We will isolate two lemmas for a later use. The first formulates the particular conclusion of Wasserstein stability we will need for the final proof.

\begin{lem}\label{lem:Wassfinal}
For \(n\geq 1\), let \(\left(M_{n},d_{n}\right)\) be a compact metric space, and suppose we are given a measurable family \(\{\alpha_{n,z}:z\in Z_{n}\}\) of probability measures on \(M_{n}\) defined on a probability space \(\left(Z_{n},\zeta_{n}\right)\) and a probability measure \(\beta_{n}\) on \(M_{n}\). Assume that the sequence of average Wasserstein distances satisfies
\[\int_{Z_{n}}\mathbf{W}_{d_{n}}\left(\alpha_{n,z},\beta_{n}\right)d\zeta_{n}\left(z\right)\to0\,\,\text{ as}\,\, n\to\infty.\]
Then for every measurable family \(\{f_{n,z}:M_{n}\to\mathbb{R}:z\in Z_{n},n\geq1\}\) of uniformly bounded uniformly equicontinuous functions, it holds that
\[\int_{Z_{n}}\Big|\int_{M_{n}}f_{n,z}d\alpha_{n,z}-\int_{M_{n}}f_{n,z}d\beta_{n}\Big|d\zeta_{n}\left(z\right)\to0\text{ as }n\to\infty.\]
\end{lem}

\begin{proof}[Proof of \cref{lem:Wassfinal}]
Put \(C:=\sup_{n\geq 1}\sup_{z\in Z_{n}}\left\|f_{n,z}\right\|_{\infty}<+\infty\). Fix \(\epsilon>0\). Choose \(\delta>0\) such that \(d_{n}\left(x,y\right)<\delta\implies\left|f_{n,z}\left(x\right)-f_{n,z}\left(y\right)\right|<\epsilon\) for all \(n,z\). Put \(L:=2C/\delta\), and for each \(n,z\) define
\[F_{n,z}\left(x\right):=\inf_{y\in M_{n}}\left(f_{n,z}\left(y\right)+Ld_{n}\left(x,y\right)\right).\]
Then \(F_{n,z}\) is \(L\)-Lipschitz and \(F_{n,z}\leq f_{n,z}\), by taking \(y=x\) in the infimum. We claim that \(F_{n,z}\geq f_{n,z}-\epsilon\); indeed, fix \(x\in M_{n}\), and for \(y\in M_{n}\) consider two alternatives:
\[\begin{cases}
f_{n,z}\left(y\right)+Ld_{n}\left(x,y\right)\geq f_{n,z}\left(y\right)\geq f_{n,z}\left(x\right)-\epsilon
& d_{n}\left(x,y\right)<\delta\\
f_{n,z}\left(y\right)+Ld_{n}\left(x,y\right)\geq-C+L\delta=C\geq f_{n,z}\left(x\right)\geq f_{n,z}\left(x\right)-\epsilon
& d_{n}\left(x,y\right)\geq\delta.
\end{cases}\]
Taking infimum over \(y\in M_{n}\), we get the inequality. We thus found that \(\left\|F_{n,z}-f_{n,z}\right\|_{\infty}\leq\epsilon\), and therefore
\begin{align*}
\Big|\int_{M_{n}}f_{n,z}d\alpha_{n,z}-\int_{M_{n}}f_{n,z}d\beta_{n}\Big|
&\leq2\epsilon+\Big|\int_{M_{n}}F_{n,z}d\alpha_{n,z}-\int_{M_{n}}F_{n,z}d\beta_{n}\Big|\\
&=2\epsilon+L\cdot\Big|\int_{M_{n}}\frac{1}{L}F_{n,z}d\alpha_{n,z}-\int_{M_{n}}\frac{1}{L}F_{n,z}d\beta_{n}\Big|\leq2\epsilon+L\mathbf{W}_{d_{n}}\left(\alpha_{n,z},\beta_{n}\right).
\end{align*}
Integrating over \(Z_{n}\), taking limsup along \(n\to\infty\), and then letting \(\epsilon\to0\), complete the proof.
\end{proof}

Our general strategy is to show that every nonsingular Borel action of a Wasserstein group is measure preserving. At this point one could finish the proof by appealing to Glasner--Tsirelson--Weiss' theorem \cite[Thm.~1.1]{glasner2005automorphism}. However, this would be an overkill, since one can apply instead the following elementary fact.

\begin{lem}\label{lem:strange}
Let \(G\curvearrowright\left(X,\mu\right)\) be a nonsingular Borel action, with the property that every probability measure on \(X\) mutually absolutely continuous with \(\mu\) is invariant.\footnote{Of course, this forces \(\mu\) itself to be invariant.} Then \(\mu\left(\operatorname{Fix}\left(G\curvearrowright X\right)\right)=1\).
\end{lem}

\begin{proof}[Proof of \cref{lem:strange}]
By \cref{thm:BecKec} we may assume that \(X\) is a compact \(G\)-space, and in fact we will only need that the action is continuous. Let \(\mu\) be an arbitrary quasi-invariant probability measure for \(G\curvearrowright X\), and by the assumption \(\mu\) is invariant. Fix an arbitrary \(g\in G\), and assume towards a contradiction that it does not act trivially modulo \(\mu\). Then there is a Borel set \(A\subseteq X\) with \(\mu\left(A\triangle g.A\right)>0\). Define
\[\phi:=1+\tfrac{1}{2}\left(\mathbf{1}_{A}-\mathbf{1}_{g.A}\right)\in L^{1}\left(\mu\right),\]
and let \(\nu\) be the probability measure \(d\nu=\phi d\mu\). Since \(\nu\sim\mu\), by the assumption \(\nu\) is invariant, and in particular \(\nu\left(A\right)=\nu\left(g.A\right)\). On the other hand, by definition of \(\nu\) we have
\begin{align*}
\nu\left(A\right)-\nu\left(g.A\right)
&=\int_{X}\left(\mathbf{1}_{A}-\mathbf{1}_{g.A}\right)\phi d\mu=\mu\left(A\right)-\mu\left(g.A\right)+\frac{1}{2}\int_{X}\left(\mathbf{1}_{A}-\mathbf{1}_{g.A}\right)^{2}d\mu\\
&=\frac{1}{2}\int_{X}\left(\mathbf{1}_{A}-\mathbf{1}_{g.A}\right)^{2}d\mu=\frac{1}{2}\mu\left(A\triangle g.A\right)>0,
\end{align*}
where the third equality uses the invariance of \(\mu\). This is a contradiction, so we conclude that there exists a Borel set \(X_{g}\subseteq X\) with \(\mu\left(X_{g}\right)=1\) such that \(g.x=x\) for every \(x\in X_{g}\). Fix a countable dense set \(D\subset G\). Since the action is continuous we have \(\bigcap_{g\in D}X_{g}\subseteq\operatorname{Fix}\left(G\curvearrowright X\right)\), and we conclude that
\[\mu\big(\operatorname{Fix}\left(G\curvearrowright X\right)\big)\geq\mu\big(\bigcap\nolimits_{g\in D}X_{g}\big)=1.\qedhere\]
\end{proof}

\subsection{The case of Wasserstein groups}

We will now prove \cref{mthm1} for Wasserstein groups. Let \(G\curvearrowright\left(X,\mu\right)\) be a nonsingular Borel action of a Wasserstein group. Replacing \(\mu\) by an equivalent probability measure and using \cref{thm:BecKec}, we may assume that \(X\) is a compact \(G\)-space and that \(\mu\) is a \(G\)-quasi-invariant Borel probability measure.

\smallskip

Fix a Wasserstein skeleton
\[\mathcal{K}=K_{1}\cup K_{2}\cup\dotsm<G\]
with respect to a compatible right-invariant metric \(d\) on \(G\). Write \(d_{n}:=d\mid_{K_{n}}\), let \(m_{n}\) be Haar probability measure on \(K_{n}\), and let \(r_{n}\to0\) be a sequence witnessing that \(\left(K_{n},d_{n},m_{n}\right)_{n\geq 1}\) is a Wasserstein family.

\smallskip

For each \(n\geq 1\) define the following objects:
\begin{itemize}
    \item Let \(\mathcal{I}_{n}\) denote the \(K_{n}\)-invariant sub-\(\sigma\)-algebra of \(X\) modulo \(\mu\), and denote
    \[\mathsf{Q}_{n}\left(\cdot\right):=\mathbb{E}_{\mu}\left[\cdot\mid\mathcal{I}_{n}\right].\]
    \item Using the Mackey cocycle theorem, choose a pointwise version of the Radon--Nikodym cocycle of the nonsingular \(K_{n}\)-subaction,
    \[\nabla^{n}:K_{n}\times X\to\mathbb{R}_{>0},\quad\left(g,x\right)\mapsto\nabla_{g}^{n}\left(x\right).\]
    \item For a nonnegative Borel function \(f:X\to\mathbb{R}\), define
    \[\mathsf{D}_{n}f\left(x\right):=\int_{K_{n}}\nabla_{g}^{n}\left(x\right)f\left(g.x\right)dm_{n}\left(g\right),\text{ and abbreviate }\mathsf{D}_{n}:=\mathsf{D}_{n}1.\]
    As \(\mathsf{D}_{n}>0\) \(\mu\)-a.e. and \(\int_{X}\mathsf{D}_{n}d\mu=1\) by Fubini's theorem, define a probability measure \(\nu_{n}\) by
    \[d\nu_{n}:=\mathsf{D}_{n}d\mu.\]
\end{itemize}

The following claim is a standard computation of conditional expectations; see e.g. \cite[Prop.~5.1]{avraham2024uncountable}.

\begin{claim}\label{clm:compavg}
The measure \(\nu_{n}\) is \(K_{n}\)-invariant, and for every nonnegative Borel function \(f:X\to\mathbb{R}\),
\[\mathsf{Q}_{n}f=\frac{\mathsf{D}_{n}f}{\mathsf{D}_{n}}\quad\mu\text{-a.e.}\]
Moreover, for every nonnegative \(K_{n}\)-invariant function \(F:X\to\mathbb{R}\),
\begin{equation}\label{eq:FDiden}
\int_{X}F\cdot\mathsf{D}_{n}fd\mu=\int_{X}F\cdot fd\mu.
\end{equation}
\end{claim}

\begin{proof}[Proof of \cref{clm:compavg}]
Let \(F:X\to\mathbb{R}\) be a nonnegative \(K_{n}\)-invariant function. Then
\begin{align*}
\int_{X}F\cdot\mathsf{D}_{n}fd\mu	
&=\int_{K_{n}}\int_{X}F\left(x\right)\cdot\nabla_{g}^{n}\left(x\right)f\left(g.x\right)d\mu\left(x\right)dm_{n}\left(g\right)\\
&=\int_{K_{n}}\int_{X}F\left(g.x\right)\cdot\nabla_{g}^{n}\left(x\right)f\left(g.x\right)d\mu\left(x\right)dm_{n}\left(g\right)\\
&=\int_{K_{n}}\int_{X}F\left(x\right)\cdot f\left(x\right)d\mu\left(x\right)dm_{n}\left(g\right)=\int_{X}F\cdot fd\mu,
\end{align*}
where the first and last equalities are by Fubini's theorem, the second equality uses that \(F\) is \(K_{n}\)-invariant, and the third equality uses the Radon--Nikodym derivative property. This establishes \eqref{eq:FDiden}.

\smallskip

Next, let us record a basic identity. For every nonnegative Borel \(f:X\to\mathbb{R}\) and every \(k\in K_{n}\), using the cocycle identity \(\nabla_{g}^{n}\left(k.x\right)=\frac{\nabla_{gk}^{n}\left(x\right)}{\nabla_{k}^{n}\left(x\right)}\) and the right-invariance of \(m_{n}\), we have
\begin{equation}\label{eq:Diden}
\mathsf{D}_{n}f\left(k.x\right)=\int_{K_{n}}\nabla_{g}^{n}\left(k.x\right)f\left(gk.x\right)dm_{n}\left(g\right)=\frac{1}{\nabla_{k}^{n}\left(x\right)}\int_{K_{n}}\nabla_{gk}^{n}\left(x\right)f\left(gk.x\right)dm_{n}\left(g\right)=\frac{\mathsf{D}_{n}f\left(x\right)}{\nabla_{k}^{n}\left(x\right)}.
\end{equation}
We now show that \(\nu_{n}\) is \(K_{n}\)-invariant. For every nonnegative Borel \(F:X\to\mathbb{R}\), using \eqref{eq:Diden} we get
\begin{align*}
\int_{X}F\left(k.x\right)d\nu_{n}\left(x\right)
&=\int_{X}F\left(k.x\right)\mathsf{D}_{n}\left(x\right)d\mu\left(x\right)\\
&=\int_{X}F\left(k.x\right)\nabla_{k}^{n}\left(x\right)\mathsf{D}_{n}\left(k.x\right)d\mu\left(x\right)\\
&=\int_{X}F\left(y\right)\mathsf{D}_{n}\left(y\right)d\mu\left(y\right)=\int_{X}F\left(y\right)d\nu_{n}\left(y\right).
\end{align*}
This shows that \(\nu_{n}\) is \(K_{n}\)-invariant. It remains to identify the conditional expectation. Fix a nonnegative Borel function \(f:X\to\mathbb{R}\), and put \(H_{f}:=\mathsf{D}_{n}f/\mathsf{D}_{n}\). Then using \eqref{eq:Diden} twice, for \(f\) and for \(1\), we get
\[\frac{\mathsf{D}_{n}f\left(k.x\right)}{\mathsf{D}_{n}\left(k.x\right)}=\frac{\mathsf{D}_{n}f\left(x\right)/\nabla_{k}^{n}\left(x\right)}{\mathsf{D}_{n}\left(x\right)/\nabla_{k}^{n}\left(x\right)}=\frac{\mathsf{D}_{n}f\left(x\right)}{\mathsf{D}_{n}\left(x\right)},\quad k\in K_{n}.\]
Therefore \(H_{f}\) is \(K_{n}\)-invariant and hence \(\mathcal{I}_{n}\)-measurable. Using \eqref{eq:FDiden} twice, for \(F\) and for \(F\cdot H_{f}\) (which is also \(\mathcal{I}_{n}\)-measurable), we get
\[\int_{X}F\cdot H_{f}d\mu=\int_{X}F\cdot H_{f}\cdot\mathsf{D}_{n}d\mu=\int_{X}F\cdot\mathsf{D}_{n}fd\mu=\int_{X}F\cdot fd\mu.\]
Since this holds for every nonnegative \(\mathcal{I}_{n}\)-measurable \(F\), and since \(H_{f}\) is \(\mathcal{I}_{n}\)-measurable, we deduce that \(H_{f}=\mathbb{E}_{\mu}\left[f\mid\mathcal{I}_{n}\right]\), which in previous notations means \(\mathsf{Q}_{n}f=\mathsf{D}_{n}f/\mathsf{D}_{n}\).
\end{proof}

Let us define some additional objects:
\begin{itemize}
    \item For a nonnegative Borel function \(f:X\to\mathbb{R}\), define the Haar averages
    \[\mathsf{R}_{n}f\left(x\right):=\int_{K_{n}}f\left(g.x\right)dm_{n}\left(g\right).\]
    (Note that when \(\mu\) is invariant, one has \(\mathsf{R}_{n}f=\mathsf{Q}_{n}f\)).
    \item Define a \(\mathcal{D}\left(K_{n},m_{n}\right)\)-valued measurable kernel \(\Phi_{n}:=\{w_{n,x}:x\in X\}\) on \(\left(X,\nu_{n}\right)\), by letting
    \[w_{n,x}\left(g\right):=\frac{\nabla_{g}^{n}\left(x\right)}{\mathsf{D}_{n}\left(x\right)}.\]
    Note that by \cref{clm:compavg} one has
    \[\mathsf{Q}_{n}f\left(x\right)=\int_{K_{n}}f\left(g.x\right)w_{n,x}\left(g\right)dm_{n}\left(g\right).\]
\end{itemize}

\begin{claim}\label{clm:locvar}
It holds that
\(\mathbf{v}_{m_{n}}\left(\Phi_{n},r_{n}\right)\to0\), and hence, since \(\left(K_{n},d_{n},m_{n}\right)_{n\geq 1}\) is a Wasserstein family,
\[\mathbf{w}_{d_{n}}\left(\Phi_{n}\right)\to0.\]
\end{claim}

\begin{proof}[Proof of \cref{clm:locvar}]
Recall that Cauchy--Schwarz inequality gives
\[\mathbf{V}_{m}\left(\phi,\phi'\right):=\frac{1}{2}\mathbb{E}\left[\left|\phi-\phi'\right|\right]\leq\big\|\sqrt{\phi}-\sqrt{\phi'}\big\|_{L^{2}\left(m\right)}.\]
Then for every \(h\in K_{n}\),
\[\int_{X}\mathbf{V}_{m_{n}}\left(h.w_{n,x},w_{n,x}\right)d\nu_{n}\left(x\right)\leq\Big(\int_{X}\int_{K_{n}}\Big(\sqrt{w_{n,x}\left(h^{-1}g\right)}-\sqrt{w_{n,x}\left(g\right)}\Big)^{2}dm_{n}\left(g\right)d\nu_{n}\left(x\right)\Big)^{1/2}.\]
Let us compute the square of the right-hand side with the cocycle identity \(\nabla_{h^{-1}g}^{n}\left(x\right)=\nabla_{h^{-1}}^{n}\left(g.x\right)\nabla_{g}^{n}\left(x\right)\):
\begin{align*}
&\int_{X}\int_{K_{n}}\Big(\sqrt{w_{n,x}\left(h^{-1}g\right)}-\sqrt{w_{n,x}\left(g\right)}\Big)^{2}dm_{n}\left(g\right)d\nu_{n}\left(x\right)\\
&\quad=\int_{K_{n}}\int_{X}\Big(\sqrt{\nabla_{h^{-1}}^{n}\left(g.x\right)}-1\Big)^{2}\nabla_{g}^{n}\left(x\right)d\mu\left(x\right)dm_{n}\left(g\right)\\
&\quad=\int_{K_{n}}\int_{X}\Big(\sqrt{\nabla_{h^{-1}}^{n}\left(y\right)}-1\Big)^{2}d\mu\left(y\right)dm_{n}\left(g\right)\\
&\quad=\left\|\kappa_{h^{-1}}1-1\right\|_{L^{2}\left(\mu\right)}^{2},
\end{align*}
where \(\kappa\) is the nonsingular Koopman representation of \(G\curvearrowright\left(X,\mu\right)\). By \cref{fct:koopmancont}, since \(r_{n}\to0\) one has
\[\sup\nolimits_{h\in B_{d_{n}}\left(r_{n}\right)}\left\|\kappa_{h^{-1}}1-1\right\|_{L^{2}\left(\mu\right)}\to0.\]
Combining both estimates gives precisely \(\mathbf{v}_{m_{n}}\left(\Phi_{n},r_{n}\right)\to0\).
\end{proof}

\begin{claim}\label{clm:QRclos}
For every \(\varphi\in C\left(X\right)\),
\[\left\|\mathsf{Q}_{n}\varphi-\mathsf{R}_{n}\varphi\right\|_{L^{1}\left(\nu_{n}\right)}\to0.\]
\end{claim}

\begin{proof}[Proof of \cref{clm:QRclos}]
For \(n\geq 1\) and \(x\in X\) define
\[f_{n,x}:K_{n}\to\mathbb{R},\quad f_{n,x}\left(g\right):=\varphi\left(g.x\right).\]
The family \(\left\{f_{n,x}:n\geq1,\ x\in X\right\}\) is uniformly bounded. It is also uniformly equicontinuous; indeed, by continuity of the action \(G\curvearrowright X\), the compactness of \(X\), and the continuity of \(\varphi\), for every \(\epsilon>0\) there is \(\delta>0\) such that
\[d\left(k,e\right)<\delta\quad\Longrightarrow\quad\sup_{y\in X}\left|\varphi\left(k.y\right)-\varphi\left(y\right)\right|<\epsilon.\]
If \(g_{1},g_{2}\in K_{n}\) and \(d_{n}\left(g_{1},g_{2}\right)<\delta\), then \(d\left(g_{1}g_{2}^{-1},e\right)<\delta\), and with \(k=g_{1}g_{2}^{-1}\), \(y=g_{2}.x\), we obtain
\[\left|f_{n,x}\left(g_{1}\right)-f_{n,x}\left(g_{2}\right)\right|=\left|\varphi\left(k.y\right)-\varphi\left(y\right)\right|<\epsilon.\]
It follows from \cref{clm:locvar} together with \cref{lem:Wassfinal} that
\[\left\|\mathsf{Q}_{n}\varphi-\mathsf{R}_{n}\varphi\right\|_{L^{1}\left(\nu_{n}\right)}=\int_{X}\Big|\int_{K_{n}}f_{n,x}\left(g\right)w_{n,x}\left(g\right)dm_{n}\left(g\right)-\int_{K_{n}}f_{n,x}\left(g\right)dm_{n}\left(g\right)\Big|d\nu_{n}\left(x\right)\to0.\qedhere\]
\end{proof}

\begin{claim}\label{clm:weakconv}
\(\nu_{n}\Rightarrow\mu\) weakly, and consequently \(\mu\) is \(G\)-invariant.
\end{claim}

Before the proof, recall that if \(G\curvearrowright X\) is a compact \(G\)-space, then the pushforward action \(G\curvearrowright P\left(X\right)\) on the space of probability measures on \(X\) is continuous with the weak topology; indeed, if \(g_{i}\to g\) in \(G\) and \(\mu_{i}\Rightarrow\mu\) weakly in \(P\left(X\right)\), then for every \(\varphi\in C\left(X\right)\),
\[\Big|\int_{X}\varphi dg_{i\ast}\mu_{i}-\int_{X}\varphi dg_{\ast}\mu\Big|\leq\sup_{x\in X}\left|\varphi\left(g_{i}.x\right)-\varphi\left(g.x\right)\right|+\Big|\int_{X}\varphi\left(g.x\right)d\mu_{i}\left(x\right)-\int_{X}\varphi\left(g.x\right)d\mu\left(x\right)\Big|.\]
The first term tends to \(0\) by continuity of the action \(G\curvearrowright X\) and the compactness of \(X\), and the second tends to \(0\) by the weak convergence \(\mu_{i}\Rightarrow\mu\).

\begin{proof}[Proof of \cref{clm:weakconv}]
Let \(\varphi\in C\left(X\right)\). Since \(\nu_{n}\) is \(K_{n}\)-invariant by \cref{clm:compavg},
\[\int_{X}\mathsf{R}_{n}\varphi d\nu_{n}=\int_{X}\varphi d\nu_{n}.\]
On the other hand, we have
\[\int_{X}\mathsf{Q}_{n}\varphi d\nu_{n}=\int_{X}\mathsf{Q}_{n}\varphi\cdot\mathsf{D}_{n}d\mu=\int_{X}\mathsf{Q}_{n}\varphi d\mu=\int_{X}\varphi d\mu,\]
where the second equality is by \eqref{eq:FDiden} with \(f\equiv 1\), and the third equality is by the definition of \(Q_{n}\). Hence, using \cref{clm:QRclos},
\[\Big|\int_{X}\varphi d\nu_{n}-\int_{X}\varphi d\mu\Big|=\Big|\int_{X}\mathsf{R}_{n}\varphi d\nu_{n}-\int_{X}\mathsf{Q}_{n}\varphi d\nu_{n}\Big|\leq\left\|\mathsf{R}_{n}\varphi-\mathsf{Q}_{n}\varphi\right\|_{L^{1}\left(\nu_{n}\right)}\to0.\]
This shows the desired weak convergence.

\smallskip

We now pass the invariance to the limit. Let \(h\in\mathcal{K}=K_{1}\cup K_{2}\cup\dotsm\), the Wasserstein skeleton of \(G\), and let \(\varphi\in C\left(X\right)\) be arbitrary. Let \(N\geq1\) be such that \(h\in K_{n}\) for all \(n\geq N\). For every \(n\geq N\), since \(\nu_{n}\) is \(K_{n}\)-invariant, we have \(h_{\ast}\nu_{n}\left(\varphi\right)=\nu_{n}\left(\varphi\right)\). Letting \(n\to\infty\), we obtain \(h_{\ast}\mu\left(\varphi\right)=\mu\left(\varphi\right)\). This shows that \(\mu\) is \(\mathcal{K}\)-invariant. Finally, since \(G\curvearrowright P\left(X\right)\) is continuous, as we recalled above, the stabilizer of \(\mu\) is a closed subgroup of \(G\) containing the dense subgroup \(\mathcal{K}\), hence this stabilizer is \(G\), so \(\mu\) is \(G\)-invariant.
\end{proof}

Altogether, we proved that for any nonsingular Borel action \(G\curvearrowright\left(X,\mu\right)\) where \(\mu\) is a probability measure, \(\mu\) is necessarily invariant. Applying this assertion to every probability measure equivalent to \(\mu\), since such a measure is again quasi-invariant for the same Borel action, we get that every probability measure equivalent to \(\mu\) is invariant. Consequently, by \cref{lem:strange} we obtain \(\mu\left(\operatorname{Fix}\left(G\curvearrowright X\right)\right)=1\).

\subsection{The general case of locally Wasserstein groups}

We will now complete the proof of \cref{mthm1} by extending it from Wasserstein groups to locally Wasserstein groups.

\smallskip

Let \(G\curvearrowright\left(X,\mu\right)\) be a nonsingular Borel action of a locally Wasserstein group. Replacing \(\mu\) by an equivalent probability measure and using \cref{thm:BecKec}, we may assume that \(X\) is a compact \(G\)-space, and that \(\mu\) is a quasi-invariant Borel probability measure.

\smallskip

Let \(\iota_{n}:H_{n}\to G\), \(\geq1\), witness that \(G\) is locally Wasserstein. For every \(n\geq 1\), the subaction
\[H_{n}\curvearrowright\left(X,\mu\right),\quad h.x:=\iota_{n}\left(h\right).x,\]
is nonsingular spatial. Since \(H_{n}\) is a Wasserstein group, by the previous proof \(\mu\left(\operatorname{Fix}\left(H_{n}\curvearrowright X\right)\right)=1\), and therefore
\[\mu\left(X_{o}\right)=1,\quad X_{o}:=\bigcap\nolimits_{n\geq1}\operatorname{Fix}\left(H_{n}\curvearrowright X\right).\]
For every \(x\in X_{o}\), its stabilizer \(G_{x}:=\left\{g\in G:g.x=x\right\}\) is closed and contains the dense subgroup generated by \(\bigcup_{n\geq1}\iota_{n}\left(H_{n}\right)\), hence \(G_{x}=G\). It then follows that
\(\mu\left(\operatorname{Fix}\left(G\curvearrowright X\right)\right)\geq\mu\left(X_{o}\right)=1\).

\appendix

\section{Examples of Wasserstein groups}\label[appendix]{app:Wass}

In this appendix we verify the Wasserstein property for some L\'{e}vy groups appearing in Glasner--Tsirelson--Weiss, Pestov's monograph, and Giordano--Pestov; see \cite[\S1]{glasner2005automorphism}, \cite[Ch.~4]{Pestov2006}, and \cite{giordano2007some}. To this end we will use \cref{mthm2}. That is, for a Polish group \(G\) with a compatible right-invariant metric \(d\), we find a skeleton \(K_{1}\leq K_{2}\leq\dotsm<G\) such that each \(K_{n}\) admits a martingale chain
\[\left\{e\right\}=H_{n,0}\leq H_{n,1}\leq\dotsm\leq H_{n,M_{n}}=K_{n}\]
having quotient radii
\[\rho_{n,i}:=\sup_{g\in H_{n,i}}\inf_{h\in H_{n,i-1}}d\left(e,gh\right),\quad1\leq i\leq M_{n},\]
and corresponding length and scale
\[\ell_{n}:=\sum\nolimits_{i=1}^{M_{n}}\rho_{n,i}\quad\text{and}\quad\varrho_{n}:=\max_{1\leq i\leq M_{n}}\rho_{n,i},\]
so that
\[\sup\nolimits_{n\geq1}\ell_{n}<+\infty\quad\text{and}\quad\varrho_{n}\to0.\]

In practice, for each \(n\) we will specify numbers \(a_{n,1},\dotsc,a_{n,M_{n}}\), such that every coset in \(H_{n,i}/H_{n,i-1}\) has a representative \(s\in H_{n,i}\) satisfying \(d\left(e,s\right)\leq a_{n,i}\). This means \(\rho_{n,i}\leq a_{n,i}\), so it is enough to verify
\[\sup\nolimits_{n\geq1}\sum\nolimits_{i=1}^{M_{n}}a_{n,i}<+\infty\quad\text{and}\quad\max_{1\leq i\leq M_{n}}a_{n,i}\to0.\]

\subsection{A convenient metric for weak topologies}

In most of the forthcoming examples, we make use of the following particularly convenient metric for the weak topology of nonsingular automorphisms. For an automorphism \(T\) of \(\left(X,\mu\right)\), denote \(\operatorname{supp}\left(T\right):=\{x\in X:Tx\neq x\}\), which is well-defined modulo \(\mu\).

\begin{lem}\label{lem:suppweak}
Let \(\left(X,\mu\right)\) be a standard nonatomic probability space. Then \(\mathrm{Aut}^{\ast}\left(X,\mu\right)\), with the weak topology, admits a compatible right-invariant metric \(d\) such that
\[d\left(e,T\right)\leq\mu\left(\operatorname{supp}\left(T\right)\right)\text{ for every }T\in\mathrm{Aut}^{\ast}\left(X,\mu\right).\]
Since \(\mathrm{Aut}\left(Y,\nu\right)\), the group of measure preserving transformations of a standard finite or infinite \(\sigma\)-finite measure space \(\left(Y,\nu\right)\), is a closed subgroup of \(\mathrm{Aut}^{\ast}\left(X,\mu\right)\), the same holds for \(\mathrm{Aut}\left(Y,\nu\right)\) as well.
\end{lem}

\begin{proof}[Proof of \cref{lem:suppweak}]
Let \(\kappa:\mathrm{Aut}^{\ast}\left(X,\mu\right)\to\mathrm{Iso}\left(L^{1}\left(X,\mu\right)\right)\) be the homomorphism
\[\kappa_{T}f:=\frac{dT_{\ast}\mu}{d\mu}\cdot f\circ T^{-1}.\]
The weak topology is induced by the strong operator topology via \(\kappa\); see \cite[\S4.1]{giordano2007some}. Choose an \(L^{1}\)-dense sequence \(\left(f_{n}\right)_{n\geq1}\subseteq L^{\infty}\left(X,\mu\right)\) and positive sequence \(\left(c_{n}\right)_{n\geq 1}\) with \(\sum\nolimits_{n\geq1}2c_{n}\left\|f_{n}\right\|_{\infty}\leq1\). Define
\[d\left(S,T\right):=\sum\nolimits_{n\geq1}c_{n}\cdot\min\big\{1,\left\|\kappa_{ST^{-1}}f_{n}-f_{n}\right\|_{L^{1}\left(\mu\right)}\big\}.\]
Then \(d\) is a compatible right-invariant metric for the weak topology. Let \(T\in\mathrm{Aut}^{\ast}\left(X,\mu\right)\). For every \(f\in L^{\infty}\left(X,\mu\right)\), since \(\operatorname{supp}\left(T\right)\) is \(T\)-invariant \(\kappa_{T}f-f\) vanishes on \(X\backslash\operatorname{supp}\left(T\right)\), and hence
\[\left\|\kappa_{T}f-f\right\|_{L^{1}\left(\mu\right)}\leq\int_{\operatorname{supp}\left(T\right)}\left|\kappa_{T}f\right|d\mu+\int_{\operatorname{supp}\left(T\right)}\left|f\right|d\mu=\int_{\operatorname{supp}\left(T\right)}\left|f\right|d\mu+\int_{\operatorname{supp}\left(T\right)}\left|f\right|d\mu\leq2\left\|f\right\|_{\infty}\mu\left(\operatorname{supp}\left(T\right)\right).\]
It follows that
\[d\left(e,T\right)\leq\sum\nolimits_{n\geq1}2c_{n}\left\Vert f_{n}\right\Vert_{\infty}\mu\left(\operatorname{supp}\left(T\right)\right)\leq\mu\left(\text{supp}\left(T\right)\right).\qedhere\]
\end{proof}

\subsection{\(L^{0}\)-groups with compact targets}

Let \(\left(Y,\eta\right)\) be a standard nonatomic probability space and let \(K\) be a compact metrizable group. Choose a compatible bi-invariant metric \(\delta\) on \(K\) with \(\operatorname{diam}\left(K,\delta\right)\leq1\), and equip \(L^{0}\left(\left(Y,\eta\right),K\right)\) with the compatible right-invariant metric
\[d\left(f_{1},f_{2}\right):=\int_{Y}\delta\left(f_{1}\left(y\right),f_{2}\left(y\right)\right)d\eta\left(y\right).\]
The L\'{e}vy property of these groups is classical; see \cite[\S A.5]{glasner2005automorphism}, \cite[\S4.2]{Pestov2006}, \cite[Corrs.~2.10--2.11]{giordano2007some}.

\smallskip

Let \(\left(\mathcal{P}_{n}\right)_{n\geq1}\) be a refining sequence of finite Borel partitions of \(Y\),
\[\mathcal{P}_{n}=\left\{P_{n,1},\dotsc,P_{n,N_{n}}\right\},\]
which generates the measure algebra of \(\left(Y,\eta\right)\) and satisfies \(\max_{i}\eta\left(P_{n,i}\right)\to0\). Let \(K_{n}\) be the compact subgroup of functions which are constant on each atom of \(\mathcal{P}_{n}\). Then \(K_{n}\cong K^{\mathcal{P}_{n}}\), the groups \(K_{n}\) are increasing, and their union is dense in \(L^{0}\left(\left(Y,\eta\right),K\right)\).

\smallskip

For \(0\leq i\leq N_{n}\), let \(H_{n,i}<K_{n}\) be the subgroup supported on \(P_{n,1}\cup\dotsm\cup P_{n,i}\). Every coset in \(H_{n,i}/H_{n,i-1}\) has a representative supported on \(P_{n,i}\), and therefore of \(d\)-distance at most
\[a_{n,i}:=\eta\left(P_{n,i}\right)\]
from the identity. Thus \(\rho_{n,i}\leq a_{n,i}\), and hence
\[\ell_{n}\leq\sum\nolimits_{i=1}^{N_{n}}a_{n,i}=1\quad\text{and}\quad\varrho_{n}\leq\max_{1\leq i\leq N_{n}}a_{n,i}\to0.\]
Therefore, \(L^{0}\left(\left(Y,\eta\right),K\right)\) is a Wasserstein group by \cref{mthm2}.

\subsection{Measure preserving automorphism groups}

Let \(\left(X,\mu\right)\) be a standard nonatomic finite or infinite \(\sigma\)-finite measure space. The group \(\mathrm{Aut}\left(X,\mu\right)\), with the weak topology, is a classical L\'{e}vy group; see \cite[\S1]{glasner2005automorphism}, \cite[\S4.5]{Pestov2006}, \cite[Thm.~4.2]{giordano2007some}.

\smallskip

First assume that \(\mu\) is a probability measure. let \(d\) be a metric on \(\mathrm{Aut}\left(X,\mu\right)\) as in \cref{lem:suppweak}. Identify \(\left(X,\mu\right)\) with \(\left(\left[0,1\right],\lambda\right)\). Let \(\mathcal{D}_{n}\) be the dyadic partition into \(N_{n}:=2^{n}\) intervals, and let \(K_{n}\cong S_{N_{n}}\) be the finite group of interval exchange transformations which permute these intervals by translations. The groups \(K_{n}\) are increasing, and \(\bigcup_{n\geq1}K_{n}\) is dense in \(\mathrm{Aut}\left(\left[0,1\right],\lambda\right)\) (see \cite[Thm.~4.1]{giordano2007some}).

\smallskip

For \(0\leq i\leq N_{n}\), let \(H_{n,i}<K_{n}\) be the subgroup which permutes the first \(i\) dyadic intervals and fixes the rest. Every coset in \(H_{n,i}/H_{n,i-1}\) has a representative among
\[e,\left(1\,i\right),\dotsc,\left(i-1\,i\right),\]
and each such representative is supported on at most two dyadic intervals. Thus, with
\[a_{n,i}:=2/N_{n},\]
we have \(\rho_{n,i}\leq a_{n,i}\). Hence
\[\ell_{n}\leq\sum\nolimits_{i=1}^{N_{n}}a_{n,i}=2\quad\text{and}\quad\varrho_{n}\leq\max_{1\leq i\leq N_{n}}a_{n,i}=2/N_{n}\to0.\]
Therefore, when \(\mu\) is a probability measure, \(\mathrm{Aut}\left(X,\mu\right)\) is a Wasserstein group by \cref{mthm2}.

\smallskip

If \(\mu\) is infinite \(\sigma\)-finite, identify \(\left(X,\mu\right)\) with \(\left(\left[0,\infty\right),\lambda\right)\). Choose a probability measure \(\mu_{o}\) equivalent to \(\mu\), and let \(d\) be a metric on \(\mathrm{Aut}\left(X,\mu\right)\) as in \cref{lem:suppweak}. For \(n\geq1\), let \(\mathcal{P}_{n}\) be the partition of \(\left[0,n\right)\) into \(N_{n}=n2^{n}\) intervals of \(\mu\)-measure \(2^{-n}\). Let \(K_{n}\cong S_{N_{n}}\) be the finite group of interval exchange transformations which permute the atoms of \(\mathcal{P}_{n}\) by translations and fix \(\left[n,\infty\right)\) pointwise. The groups \(K_{n}\) are increasing, and by \cite[Thm.~4.1]{giordano2007some}, \(\bigcup_{n\geq1}K_{n}\) is dense in \(\mathrm{Aut}\left(X,\mu\right)\). Write
\[\mathcal{P}_{n}=\left\{P_{n,1},\dotsc,P_{n,N_{n}}\right\},\quad p_{n,i}:=\mu_{o}\left(P_{n,i}\right),\]
and order the atoms so that \(p_{n,1}\leq\dotsm\leq p_{n,N_{n}}\). Since \(\mu_{o}\ll\mu\) and \(\mu\left(P_{n,i}\right)=2^{-n}\) for all \(i\), we have \(\max_{0\leq i\leq N_{n}}p_{n,i}\to0\). For \(0\leq i\leq N_{n}\), let \(H_{n,i}<K_{n}\) be the subgroup which permutes \(P_{n,1},\dotsc,P_{n,i}\) and fixes the rest. Every coset in \(H_{n,i}/H_{n,i-1}\) has a representative which is either the identity or a transposition between \(P_{n,i}\) and some \(P_{n,j}\), \(j<i\). Such a representative is supported on \(P_{n,i}\cup P_{n,j}\), and hence its \(d\)-distance from the identity is at most
\(p_{n,i}+p_{n,j}\leq2p_{n,i}\). Thus, with
\[a_{n,i}:=2p_{n,i},\]
we have \(\rho_{n,i}\leq a_{n,i}\), and therefore
\[\ell_{n}\leq\sum\nolimits_{i=1}^{N_{n}}a_{n,i}=2\mu_{o}\left(\left[0,n\right)\right)\leq2\quad\text{and}\quad\varrho_{n}\leq\max_{1\leq i\leq N_{n}}a_{n,i}\to0.\]
Therefore, also when \(\mu\) is infinite \(\sigma\)-finite, \(\mathrm{Aut}\left(X,\mu\right)\) is a Wasserstein group by \cref{mthm2}.

\subsection{Amenable full groups}\label{sct:fullamen}

Let \(\mathcal{R}\) be a nonsingular ergodic amenable countable Borel equivalence relation on a standard nonatomic probability space \(\left(X,\mu\right)\). The full group \(\left[\mathcal{R}\right]\) consists of the nonsingular automorphisms whose graph is contained in \(\mathcal{R}\). Equip \(\left[\mathcal{R}\right]\) with the right-invariant uniform metric
\[d\left(S,T\right):=\mu\left(\left\{x:S^{-1}x\neq T^{-1}x\right\}\right).\]

Giordano--Pestov prove that \(\left[\mathcal{R}\right]\) is a L\'{e}vy group precisely when \(\mathcal{R}\) is amenable; see \cite[Thm.~5.7]{giordano2007some}, \cite[\S4.6]{Pestov2006}. By the Connes--Feldman--Weiss theorem, as in the concrete model in \cite[\S5.2]{giordano2007some}, we may assume that
\[X=\prod\nolimits_{n\geq 1}\{0,1\},\]
that \(\mathcal{R}\) is the tail equivalence relation,
\[\left(x,y\right)\in\mathcal{R}\iff\exists_{N=N\left(x,y\right)\geq 1},\,\forall_{n\geq N},\,x_{n}=y_{n},\]
and that \(\mu\) is a nonatomic probability measure which is quasi-invariant to \(\mathcal{R}\). For \(n\geq1\), let \(\Omega_{n}:=\left\{0,1\right\}^{n}\) and let \(\mathcal{P}_{n}\) be the partition into \(n\)-cylinders
\[P_{n,\omega}:=\left\{x\in X:\left(x_{1},\dotsc,x_{n}\right)=\omega\right\},\quad\omega\in\Omega_{n}.\]
Let \(K_{n}\cong S_{\Omega_{n}}\) be the finite group of permutation of \(\Omega_{n}\), acting on \(X\) by permuting the first \(n\) coordinates. Since each element of \(K_{n}\) changes only finitely many coordinates we have \(K_{n}<\left[\mathcal{R}\right]\). The groups \(K_{n}\) are increasing, and by \cite[Lem.~5.2]{giordano2007some}, \(\bigcup_{n\geq1}K_{n}\) is uniformly dense in \(\left[\mathcal{R}\right]\).

\smallskip

Write \(p_{n,\omega}:=\mu\left(P_{n,\omega}\right)\). Since the cylinder partitions generate the Borel structure and \(\mu\) is nonatomic,
\[\max\nolimits_{\omega\in\Omega_{n}}p_{n,\omega}\to0.\]
Indeed, otherwise there are \(\epsilon>0\), integers \(n_{k}\to\infty\), and cylinders \(P_{n_{k},\omega_{k}}\) with \(\mu\left(P_{n_{k},\omega_{k}}\right)\geq\epsilon\). By compactness of \(X\), after passing to a subsequence, the words \(\omega_{k}\) converge to some \(x\in X\). Hence for every \(m\geq1\), eventually \(P_{n_{k},\omega_{k}}\subseteq P_{m,\left(x_{1},\dotsc,x_{m}\right)}\), and therefore \(\mu\left(P_{m,\left(x_{1},\dotsc,x_{m}\right)}\right)\geq\epsilon\). By continuity \(\mu\left(\{x\}\right)\geq\epsilon\), contradicting non-atomicity.

\smallskip

Let \(N_{n}=2^{n}\), and enumerate the cylinders \(P_{n,1},\dotsc,P_{n,N_{n}}\) in such a way that
\[p_{n,1}\leq\dotsm\leq p_{n,N_{n}}.\]
For \(0\leq i\leq N_{n}\), let \(H_{n,i}<K_{n}\) be the subgroup which permutes \(P_{n,1},\dotsc,P_{n,i}\) and fixes the rest. Let \(uH_{n,i-1}\in H_{n,i}/H_{n,i-1}\). If \(u\left(P_{n,i}\right)=P_{n,j}\), then \(j\leq i\). If \(j=i\), the coset has the identity as a representative. If \(j<i\), let \(s\) be the transposition of the two cylinders \(P_{n,i}\) and \(P_{n,j}\), fixing all other \(n\)-cylinders. Then \(s^{-1}u\in H_{n,i-1}\), and hence \(s\) is a representative of the coset \(uH_{n,i-1}\). Moreover, \(s\) is supported on \(P_{n,i}\cup P_{n,j}\), and therefore \(d\left(e,s\right)\leq p_{n,i}+p_{n,j}\leq2p_{n,i}\). Then with
\[a_{n,i}:=2p_{n,i},\]
we have \(\rho_{n,i}\leq a_{n,i}\), hence
\[\ell_{n}\leq\sum\nolimits_{i=1}^{N_{n}}a_{n,i}=2\quad\text{and}\quad\varrho_{n}\leq\max_{1\leq i\leq N_{n}}a_{n,i}\to0.\]
Therefore, \(\left[\mathcal{R}\right]\) is a Wasserstein group by \cref{mthm2}.

\subsection{Nonsingular automorphism groups}

Let \(\left(X,\mu\right)\) be a standard nonatomic finite or infinite \(\sigma\)-finite nonatomic measure space. The nonsingular automorphism group \(\mathrm{Aut}^{\ast}\left(X,\mu\right)\) is a L\'{e}vy group by \cite[Thm.~6.1]{giordano2007some}; see also \cite[\S4.6]{Pestov2006}. We verify that it is Wasserstein. Note that \(\mathrm{Aut}^{\ast}\left(X,\mu\right)\cong\mathrm{Aut}^{\ast}\left(X,\nu\right)\) whenever \(\mu\sim\nu\) both nonatomic, so we may assume that \(\mu\) is a probability measure and use the metric \(d\) from \cref{lem:suppweak}.

\smallskip

By \cite[Prop.~6.2]{giordano2007some}, there is an ergodic nonsingular transformation \(T\) whose full group \(\left[T\right]\), that is the full group of the orbit equivalence relation induced from \(T\), is weakly dense in \(\mathrm{Aut}^{\ast}\left(X,\mu\right)\). Let \(\delta\) denote the uniform metric on \(\left[T\right]\),
\[\delta\left(S,R\right):=\mu\left(\left\{x:S^{-1}x\neq R^{-1}x\right\}\right).\]
By \cref{lem:suppweak} and using right-invariance of \(d\), for \(S,R\in\left[T\right]\) we have
\[d\left(S,R\right)=d\left(e,SR^{-1}\right)\leq\mu\left(\operatorname{supp}\left(SR^{-1}\right)\right)=\delta\left(S,R\right).\]
Hence \(\delta\)-convergence in \(\left[T\right]\) implies \(d\)-convergence in \(\mathrm{Aut}^{\ast}\left(X,\mu\right)\).

\smallskip

Since the orbit equivalence relation induced from \(T\) is amenable, by the construction in \cref{sct:fullamen} there is a skeleton consisting of finite subgroups
\[K_{1}\leq K_{2}\leq\dotsm<\left[T\right]\]
such that \(\bigcup_{n\geq 1}K_{n}\) is \(\delta\)-dense in \(\left[T\right]\), as well as martingale chains inside each \(K_{n}\) with quotient representatives supported on \(P_{n,i}\cup P_{n,j}\), \(j<i\). Since \(\bigcup_{n\geq 1}K_{n}\) is \(\delta\)-dense in \(\left[T\right]\) which is \(d\)-dense in \(\mathrm{Aut}^{\ast}\left(X,\mu\right)\), then \(\bigcup_{n\geq1}K_{n}\) is \(d\)-dense in \(\mathrm{Aut}^{\ast}\left(X,\mu\right)\). Thus \(K_{1}\leq K_{2}\leq\dotsm\) is a skeleton for \(\mathrm{Aut}^{\ast}\left(X,\mu\right)\).

\smallskip

Finally, we now need to compute the quotient radii with respect to \(d\). However, the same support estimates apply: if a quotient representative \(s\in K_{n}\) is supported on \(P_{n,i}\cup P_{n,j}\), \(j<i\), then
\[d\left(e,s\right)\leq\mu\left(\operatorname{supp}\left(s\right)\right)\leq p_{n,i}+p_{n,j}\leq2p_{n,i}.\]
Therefore, the same parameters \(a_{n,i}\) as in the previous proof are valid for the ambient metric \(d\), and the same estimates on the lengths and scales show that \(\mathrm{Aut}^{\ast}\left(X,\mu\right)\) is a Wasserstein group by \(\cref{cor:qutifilt}\).

\subsection{Isometry groups of \(L^{p}\)-spaces}

Let \(1\leq p<\infty\), \(p\neq2\), and let \(\left(X,\mu\right)\) be a standard nonatomic finite or infinite \(\sigma\)-finite measure space. The fact that \(\mathrm{Iso}\left(L^{p}\left(X,\mu\right)\right)\) with the strong operator topology is a L\'{e}vy group is due to Giordano--Pestov \cite[Thm.~6.6]{giordano2007some}. We verify that it is a Wasserstein group. Note that \(L^{p}\left(X,\mu\right)\cong L^{p}\left(X,\nu\right)\) whenever \(\mu\sim\nu\) both nonatomic (via multiplication with \(\left(d\nu/d\mu\right)^{1/p}\)), so we will assume that \(\mu\) is a probability measure.

\smallskip

By the Banach--Lamperti theorem (see \cite[Cor.~6.8]{giordano2007some}),
\begin{equation}\label{eq:BanLam}
\mathrm{Iso}\left(L^{p}\left(X,\mu\right)\right)\cong L^{0}\left(\left(X,\mu\right),C\right)\rtimes\mathrm{Aut}^{\ast}\left(X,\mu\right)
\end{equation}
as topological groups, where \(C=S^{1}\) in the complex case and \(C=\mathbb{Z}/2\mathbb{Z}\) in the real case. In order to identify a convenient metric, consider the skew-product homomorphism
\[\Theta:L^{0}\left(\left(X,\mu\right),C\right)\rtimes\mathrm{Aut}^{\ast}\left(X,\mu\right)\to\mathrm{Aut}^{\ast}\left(X\times C,\mu\otimes m_{C}\right),\quad\Theta\left(u,T\right)\left(x,z\right):=\left(Tx,u\left(Tx\right)z\right).\]
This is an injective homomorphism, and we claim that it is bi-continuous, thus a topological embedding. We use the standard description of the weak topology on \(\mathrm{Aut}^{\ast}\left(Y,\nu\right)\) through the associated \(L^{p}\)-isometries: to every  \(T\in\mathrm{Aut}^{\ast}\left(Y,\nu\right)\) corresponds \(V_{T}\in\mathrm{Iso}\left(L^{p}\left(Y,\nu\right)\right)\) given by \(V_{T}f\left(y\right):=\big(\frac{dT^{-1}\nu}{d\nu}\left(y\right)\big)^{1/p}f\left(Ty\right)\). For every \(S=\Theta\left(u,T\right)\), since \(m_{C}\) is invariant we have
\[\frac{dS^{-1}\left(\mu\otimes m_{C}\right)}{d\left(\mu\otimes m_{C}\right)}\left(x,z\right)=\frac{dT^{-1}\mu}{d\mu}\left(x\right).\]
Denoting by \(\zeta:C\to\mathbb{C}\) the map \(\zeta\left(c\right)=c\), it follows that for every \(f\in L^{p}\left(X,\mu\right)\) and every \(k\),
\begin{equation}\label{eq:Vfunc}
V_{\Theta\left(u,T\right)}\left(f\otimes1\right)=V_{T}f\otimes1\quad\text{and}\quad V_{\Theta\left(u,T\right)}\left(f\otimes\zeta^{k}\right)=V_{T}\left(fu^{k}\right)\otimes\zeta^{k}.
\end{equation}
We prove continuity of \(\Theta\). Suppose \(\left(u_{m},T_{m}\right)\to\left(u,T\right)\), so \(u_{m}\to u\) in measure and \(V_{T_{m}}\to V_{T}\) strongly. For every \(f\in L^{p}\left(X,\mu\right)\) and every \(k\), by dominated convergence we have \(fu_{m}^{k}\to fu^{k}\in L^{p}\left(X,\mu\right)\), hence
\[\left\|V_{T_{m}}\left(fu_{m}^{k}\right)-V_{T}\left(fu^{k}\right)\right\|_{p}\leq\left\|fu_{m}^{k}-fu^{k}\right\|_{p}+\left\|V_{T_{m}}\left(fu^{k}\right)-V_{T}\left(fu^{k}\right)\right\|_{p}\to0.\]
Since the functions \(f\otimes\zeta^{k}\) are dense and \(V_{\Theta\left(u_{m},T_{m}\right)}\) are isometries, \(\Theta\left(u_{m},T_{m}\right)\to\Theta\left(u,T\right)\). We prove now continuity of \(\Theta^{-1}\) on \(\operatorname{Im}\left(\Theta\right)\). Suppose \(\Theta\left(u_{m},T_{m}\right)\to\Theta\left(u,T\right)\) in \(\mathrm{Aut}^{\ast}\left(X\times C,\mu\otimes m_{C}\right)\). By looking at \(f\otimes1\) in \eqref{eq:Vfunc}, we get \(V_{T_{m}}f\to V_{T}f\) for every \(f\in L^{p}\left(X,\mu\right)\), hence \(T_{m}\to T\in\mathrm{Aut}^{\ast}\left(X,\mu\right)\). By looking at \(1\otimes\zeta\) in \eqref{eq:Vfunc}, we get \(V_{T_{m}}u_{m}\to V_{T}u\in L^{p}\left(X,\mu\right)\). Therefore, using that \(V_{T_{m}}\) is an isometry,
\[\left\|u_{m}-u\right\|_{p}=\left\|V_{T_{m}}u_{m}-V_{T_{m}}u\right\|_{p}\leq\left\|V_{T_{m}}u_{m}-V_{T}u\right\|_{p}+\left\|V_{T}u-V_{T_{m}}u\right\|_{p}\to0.\]
Then this \(L^{p}\)-convergence implies \(u_{m}\to u\) in measure, and so \(\left(u_{m},T_{m}\right)\to\left(u,T\right)\).

\smallskip

Now that we found that \(\Theta\) is a topological embedding, let \(d\) be the pullback to \(\mathrm{Iso}\left(L^{p}\left(X,\mu\right)\right)\) via \(\Theta\) of the metric from \cref{lem:suppweak} on \(\mathrm{Aut}^{\ast}\left(X\times C,\mu\otimes m_{C}\right)\). Since \(\Theta\) is a topological embedding, \(d\) is a compatible right-invariant metric on \(\mathrm{Iso}\left(L^{p}\left(X,\mu\right)\right)\).

\smallskip

As in the previous construction, choose an
increasing finite skeleton
\[S_{\mathcal{P}_{1}}\leq S_{\mathcal{P}_{2}}\leq\dotsm<\mathrm{Aut}^{\ast}\left(X,\mu\right),\]
where each \(S_{\mathcal{P}_{n}}\) is realized by the nonsingular automorphisms permuting the atoms of \(\mathcal{P}_{n}\), in such a way that \(\left(\mathcal{P}_{n}\right)_{n\geq 1}\) is refining, generates the measure algebra, and \(\max_{1\leq i\leq N_{n}}\mu\left(P_{n,i}\right)\to0\). Then we know that \(\bigcup_{n}S_{\mathcal{P}_{n}}\) is dense in \(\mathrm{Aut}^{\ast}\left(X,\mu\right)\). Let \(C^{\mathcal{P}_{n}}\) be the compact group of \(C\)-valued functions which are constant on the atoms of \(\mathcal{P}_{n}\). Put the compact subgroup
\[K_{n}:=C^{\mathcal{P}_{n}}\rtimes S_{\mathcal{P}_{n}}.\]
Concretely, the elements of \(K_{n}\) are viewed as nonsingular automorphisms of  \(X\times C\) via
\[\left(u,\sigma\right).\left(x,z\right)=\left(\sigma x,u\left(\sigma x\right)z\right).\]
By construction we have \(K_{1}\leq K_{2}\leq\dotsm\). Since \(\bigcup_{n\geq1}C^{\mathcal{P}_{n}}\) is dense in \(L^{0}\left(\left(X,\mu\right),C\right)\) and \(\bigcup_{n\geq1}S_{\mathcal{P}_{n}}\) is dense in \(\mathrm{Aut}^{\ast}\left(X,\mu\right)\), it follows from \(\eqref{eq:BanLam}\) that \(\bigcup_{n\geq1}K_{n}\) is dense in \(\mathrm{Iso}\left(L^{p}\left(X,\mu\right)\right)\).

\smallskip

Fix \(n\geq 1\), and let us define a martingale chain in \(K_{n}\). Order the atoms of \(\mathcal{P}_{n}\) so that
\[p_{n,1}\leq\dotsm\leq p_{n,N_{n}},\quad p_{n,i}:=\mu\left(P_{n,i}\right).\]
We do this in two steps, first the \(u\)-coordinate and then the \(\sigma\)-coordinate, obtaining a martingale chain
\[\{e\}=H_{n,0}\leq H_{n,1}\leq\dotsm\leq H_{n,N_{n}}\leq H_{n,N_{n}+1}\leq\dotsm\leq H_{n,2N_{n}}=K_{n}.\]
For \(0\leq i\leq N_{n}\) let
\[M_{n,i}:=\left\{u\in C^{\mathcal{P}_{n}}:u\mid_{X\setminus\left(P_{n,1}\cup\dotsm\cup P_{n,i}\right)}\equiv 1\right\},\text{ and put } H_{n,i}:=M_{n,i}\rtimes\left\{e\right\}.\]
This forms the first step of the martingale chain. Next, for \(0\leq i\leq N_{n}\) let
\[S_{n,i}:=\left\{\sigma\in S_{\mathcal{P}_{n}}:\sigma\left(P_{n,j}\right)=P_{n,j}\text{ for every }j>i\right\},\text{ and put }H_{n,N_{n}+i}:=C^{\mathcal{P}_{n}}\rtimes S_{n,i}.\]
This gives a martingale chain on \(K_{n}\), and we now estimate the quotient radii. Since \(d\) is the pullback of the metric from \(\cref{lem:suppweak}\) through \(\Theta\), every element \(s\in K_{n}\) satisfies
\[d\left(e,s\right)\leq\left(\mu\otimes m_{C}\right)\left(\operatorname{supp}\left(\Theta\left(s\right)\right)\right).\]

\smallskip

For \(1\leq i\leq N_{n}\), a coset in \(H_{n,i}/H_{n,i-1}\) depends only on the value on \(P_{n,i}\), so it has a representative \(\left(v,e\right)\), where \(v=1\) outside \(P_{n,i}\). This representative satisfies \(\operatorname{supp}\left(\Theta\left(v,e\right)\right)\subseteq P_{n,i}\times C\), and therefore
\[d\left(e,\left(v,e\right)\right)\leq p_{n,i}.\]
Then for the first \(N_{n}\) steps we take
\[a_{n,i}:=p_{n,i},\quad 1\leq i\leq N_{n}.\]
For the next \(N_{n}\) steps, for \(1\leq i\leq N_{n}\), a coset in
\(H_{n,N_{n}+i}/H_{n,N_{n}+i-1}\) depends only on the position of the atom \(P_{n,i}\) among
\(P_{n,1},\dotsc,P_{n,i}\). Thus it has a representative \(\left(1,\tau\right)\), where \(\tau=e\), or \(\tau\) is a transposition exchanging \(P_{n,i}\) with some \(P_{n,j}\), \(j<i\). Such a representative satisfies \(\operatorname{supp}\left(\Theta\left(1,\tau\right)\right)\subseteq\left(P_{n,i}\cup P_{n,j}\right)\times C\), so
\[d\left(e,\left(1,\tau\right)\right)\leq p_{n,i}+p_{n,j}\leq2p_{n,i}.\]
Then for the second \(N_{n}\) steps we take
\[a_{n,N_{n}+i}:=2p_{n,i},\quad 1\leq i\leq N_{n}.\]
Therefore, \(\rho_{n,i}\leq a_{n,i}\) for every \(1\leq i\leq2N_{n}\) and so
\[\ell_{n}\leq\sum\nolimits_{i=1}^{2N_{n}}a_{n,i}
=\sum\nolimits_{i=1}^{N_{n}}p_{n,i}+\sum\nolimits_{i=1}^{N_{n}}2p_{n,i}=3\quad\text{ and }\quad\varrho_{n}\leq2\max_{1\leq i\leq N_{n}}p_{n,i}\to0.\]
Therefore, \(\mathrm{Iso}\left(L^{p}\left(X,\mu\right)\right)\) is a Wasserstein group by \cref{mthm2}.

\subsection{The unitary group of the hyperfinite \(\mathrm{II}_{1}\)-factor}\label{sct:unithyperf}

Let \(R\) be the hyperfinite \(\mathrm{II}_{1}\)-factor with trace \(\tau\). The strong topology on \(U\left(R\right)\) is induced by the compatible bi-invariant metric
\[d\left(u,v\right):=\tau\left(\left|u-v\right|\right).\]
The L\'{e}vy property of \(U\left(R\right)\) follows from the finite dimensional subfactor argument; see \cite[Prop.~3.5]{giordano2007some}, \cite[\S4.4]{Pestov2006}. The same finite dimensional skeleton satisfies the stronger Wasserstein criterion.

\smallskip

Choose an increasing sequence of matrix subalgebras
\[M_{k_{1}}\left(\mathbb{C}\right)\subseteq M_{k_{2}}\left(\mathbb{C}\right)\subseteq\dotsm\subseteq R\]
whose union is \(\left\|\cdot\right\|_{2}\)-dense in \(R\), and put
\[K_{n}:=U\left(M_{k_{n}}\left(\mathbb{C}\right)\right)<U\left(R\right).\]
Then \(\bigcup_{n\geq1}K_{n}\) is dense in \(U\left(R\right)\).  On \(K_{n}\) we have the normalized trace metric \(d_{n}\left(u,v\right)=\operatorname{tr}_{k_{n}}\left(\left|u-v\right|\right)\), as in \cref{exm:findimunit}. Using the martingale chain as in \cref{exm:findimunit}, every coset in \(H_{n,i}/H_{n,i-1}\) has a representative \(s\) with \(d_{n}\left(e,s\right)\leq4/k_{n}\), and thus
\[\rho_{n,i}\leq a_{n,i}:=\frac{4}{k_{n}}.\]
It follows that
\[\ell_{n}\leq\sum\nolimits_{i=1}^{k_{n}}a_{n,i}=4\quad\text{and}\quad\varrho_{n}\leq\frac{4}{k_{n}}\to0.\]
Therefore, \(U\left(R\right)\) is a Wasserstein group by \cref{mthm2}.

\section{Examples of locally Wasserstein groups}\label[appendix]{app:locWass}

In this appendix we verify the local Wasserstein property for some L\'{e}vy groups, as well as for a Polish group which is not L\'{e}vy. In all cases, our strategy is to show that the group is topologically generated by continuous images of \(L^{0}\)-groups with appropriate compact targets. It is worth mentioning that a negative result related in spirit to this strategy for \(\operatorname{Aut}\left(X,\mu\right)\) was established in \cite[Thm.~1.1]{solecki2023generic}.

\subsection{\(L^{0}\)-groups with compactly approximable target}

Let \(H\) be a Polish group admitting a skeleton
\[K_{1}\leq K_{2}\leq\dotsm<H,\]
and let \(\left(Y,\eta\right)\) be a standard nonatomic probability space. Then the group
\[L^{0}\left(\left(Y,\eta\right),H\right),\]
is Polish in the topology of convergence in measure by \cite[Prop.~7]{moore1976group}, and it is locally Wasserstein.

\smallskip

Indeed, for every \(n\) we showed that \(L^{0}\left(\left(Y,\eta\right),K_{n}\right)\) is Wasserstein. The inclusion \(K_{n}\hookrightarrow H\) induces a continuous group embedding \(L^{0}\left(\left(Y,\eta\right),K_{n}\right)\to L^{0}\left(\left(Y,\eta\right),H\right)\). The subgroup generated by the images of these homomorphisms is dense. Indeed, every element of \(L^{0}\left(\left(Y,\eta\right),H\right)\) can be approximated in measure by a simple function, and the finitely many values of this simple function can be approximated in \(H\) by elements of \(\bigcup_{n\geq1}K_{n}\). Since the \(K_{n}\)'s are increasing, these finitely many elements all belong to some common \(K_{N}\). Hence the approximating simple function belongs to the image of \(L^{0}\left(\left(Y,\eta\right),K_{N}\right)\). Therefore \(L^{0}\left(\left(Y,\eta\right),H\right)\) is locally Wasserstein.

\subsection{Unitary and orthogonal groups of Hilbert space}

The groups \(U\left(\mathcal{H}\right)\) and \(O\left(\mathcal{H}_{\mathbb{R}}\right)\) with the strong operator topology are L\'{e}vy groups by Gromov--Milman's argument \cite{gromov1983}; see \cite[\S A.4]{glasner2005automorphism}, \cite[\S4.1]{Pestov2006}. We show that they are locally Wasserstein by an argument inspired by \cite[Rem.~1.10]{glasner2005automorphism} (an idea attributed there to Lema\'{n}czyk--Parreau--Thouvenot~\cite{lemanczyk2000gaussian}) and \cite[Lem.~6.5]{Pestov2010}. The key is the following observation:

\begin{lem}\label{lem:L0unitemb}
Let \(E\) be a finite dimensional Hilbert space. Then the map
\[L^{0}\left(\left[0,1\right],U\left(E\right)\right)\to U\left(L^{2}\left(\left[0,1\right],E\right)\right),\footnote{Here \(L^{2}\left(\left[0,1\right],E\right)\) can be thought of as a direct sum of \(\left|E\right|\) copies of the Hilbert space \(L^{2}\left(\left[0,1\right],\mathbb{C}\text{ or }\mathbb{R}\right)\).}\quad\phi\mapsto M_{\phi},\]
given by
\[\left(M_{\phi}\xi\right)\left(t\right):=\phi\left(t\right)\left(\xi\left(t\right)\right),\]
is a continuous group embedding, where on the source we take the topology of convergence in measure, and on the target we take the strong operator topology.
\end{lem}

\begin{proof}[Proof of \cref{lem:L0unitemb}]
It is plainly a homomorphism. Suppose \(\phi_{n}\to\phi\) in measure. For \(\xi\in L^{2}\left(\left[0,1\right],E\right)\),
\[\left\|M_{\phi_{n}}\xi-M_{\phi}\xi\right\|_{2}^{2}=\int_{0}^{1}\left\|\left(\phi_{n}\left(t\right)-\phi\left(t\right)\right)\xi\left(t\right)\right\|^{2}dt.\]
Since \(E\) is finite dimensional, \(U\left(E\right)\) is compact. Then the functions \(t\mapsto\left\|\phi_{n}\left(t\right)-\phi\left(t\right)\right\|\) are uniformly bounded and converge to \(0\) in measure as \(n\to\infty\). By truncating \(\xi\) and using dominated convergence, this integral tends to \(0\), so \(M_{\phi_{n}}\to M_{\phi}\) in the strong operator topology. Finally, to see that it is injective, suppose that \(M_{\phi}=M_{\psi}\). Let \(e_{1},\dotsc,e_{k}\) be an orthonormal basis of \(E\), and for \(1\leq j\leq k\) let \(\xi_{j}\in L^{2}\left(\left[0,1\right],E\right)\) be the constant function \(t\mapsto e_{j}\). Then for each \(1\leq j\leq k\) we have
\[\phi\left(t\right)\left(e_{j}\right)=M_{\phi}\xi_{j}\left(t\right)=M_{\psi}\xi_{j}\left(t\right)=\psi\left(t\right)\left(e_{j}\right)\quad\text{for a.e. }t.\]
Intersecting these \(k\) conull set, we get that \(\phi=\psi\) a.e. and thus they coincide in \(L^{0}\left(\left[0,1\right],U\left(E\right)\right)\).
\end{proof}

Let us show that \(U\left(\mathcal{H}\right)\) is locally Wasserstein. The proof that \(O\left(\mathcal{H}\right)\) is locally Wasserstein is identical. Fix an arbitrary nonempty open set \(O\subset U\left(\mathcal{H}\right)\) for the strong topology. Let \(u\in O\), and so there are \(\xi_{1},\dotsc,\xi_{k}\in\mathcal{H}\) and \(\epsilon>0\) such that
\[W:=\{w\in U\left(\mathcal{H}\right):\left\|w\xi_{j}-u\xi_{j}\right\|<\epsilon,\,1\leq j\leq k\}\subseteq O.\]
Let \(F:=\mathrm{span}\{\xi_{1},\dotsc,\xi_{k}\}\), and consider the finite dimensional subspace \(E:=F+u\left(F\right)\). Let \(v\in U\left(E\right)\) be an isometry extending the isometry \(u\mid_{F}:F\to uF\). Fix a unitary isomorphism
\[J:\mathcal{H}\to L^{2}\left(\left[0,1\right],E\right),\]
which sends every \(\eta\in E\) to the constant function \(t\mapsto\eta\). To see why such \(J\) exists, note that we already have the unitary isomorphism \(\jmath:E\to C_{E}\), \(\jmath:\eta\mapsto\left(t\mapsto\eta\right)\), where \(C_{E}\subset L^{2}\left(\left[0,1\right],E\right)\) is the space of \(E\)-valued constant functions. Then \(C_{E}^{\perp}\) is an infinite dimensional separable Hilbert space and, since \(\mathcal{H}\) is infinite dimensional while \(E\) is finite dimensional, also \(E^{\perp}\) is an infinite dimensional separable Hilbert space. Therefore, we can choose a unitary isomorphism \(E^{\perp}\to C_{E}^{\perp}\), and complete \(\jmath\) on the orthogonal complement, obtaining the desired \(J\). Consider then the Wasserstein group
\[H:=L^{0}\left(\left[0,1\right],U\left(E\right)\right),\]
and by \cref{lem:L0unitemb} we have the continuous homomorphism
\[\iota:H\to U\left(\mathcal{H}\right),\quad\iota\left(\phi\right):=J^{-1}M_{\phi}J.\]
We finally claim that \(\iota\left(H\right)\cap O\neq\emptyset\). Let \(\phi\in H\) be the constant function \(t\mapsto v\). For every \(1\leq j\leq k\), since \(\xi_{j}\in F\subseteq E\), \(J\xi_{j}\in L^{2}\left(\left[0,1\right],E\right)\) is the constant function \(t\mapsto\xi_{j}\). Therefore, \(M_{\phi}J\xi_{j}\) is the constant function \(t\mapsto v\xi_{j}=u\xi_{j}\). Since \(u\xi_{j}\in uF\subseteq E\) and \(J\) sends \(u\xi_{j}\) to the constant function \(t\mapsto u\xi_{j}\), we get
\[\iota\left(\phi\right)\xi_{j}=J^{-1}M_{\phi}J\xi_{j}=u\xi_{j},\quad1\leq j\leq k.\]
Thus \(\iota\left(\phi\right)\in W\subseteq O\), and so \(\iota\left(H\right)\cap O\neq\emptyset\).

\subsection{The Cameron--Martin affine group}

Let \(\mathcal{H}\) be an infinite dimensional separable real Hilbert space. Let the Cameron--Martin affine group
\[\mathcal{H}\rtimes O\left(\mathcal{H}\right),\]
viewed as the group of affine isometries of \(\mathcal{H}\), with multiplication and action given by
\[\left(h,U\right)\left(k,V\right):=\left(h+Uk,UV\right),\quad \left(h,U\right).x:=h+Ux.\]
We show that \(\mathcal{H}\rtimes O\left(\mathcal{H}\right)\) is locally Wasserstein. For every \(c\in\mathcal{H}\), we have a natural isomorphism
\[G_{c}:=\left\{\left(c-Uc,U\right):U\in O\left(\mathcal{H}\right)\right\}\cong O\left(\mathcal{H}\right),\]
identifying the stabilizer of \(c\) with \(O\left(\mathcal{H}\right)\) via \(U\mapsto\left(c-Uc,U\right)\). Hence, \(G_{c}\) is locally Wasserstein.

\smallskip

We now use the second formulation in \cref{def:appWass}. First observe that every affine isometry belongs to a product of two such stabilizers. Indeed, for \(\xi\in\mathcal{H}\), the element \(s_{\xi}:=\left(2\xi,-I\right)\) fixes \(\xi\), and hence \(s_{\xi}\in G_{\xi}\). Therefore, for every \(\left(h,U\right)\in\mathcal{H}\rtimes O\left(\mathcal{H}\right)\) we have
\[\left(h,U\right)=s_{h/2}\left(0,-U\right)\in G_{h/2}G_{0}.\]
Let \(O\subseteq\mathcal{H}\rtimes O\left(\mathcal{H}\right)\) be a nonempty open set, and choose \(\left(h,U\right)\in O\). Find open neighborhoods \(s_{h/2}\in O_{1}\subseteq G_{h/2}\) and \(\left(0,-U\right)\in O_{0}\subseteq G_{0}\), in the subspace topologies, with \(O_{1}O_{0}\subseteq O\). Since \(G_{h/2}\) and \(G_{0}\) are locally Wasserstein, each of \(O_{1}\) and \(O_{0}\) meets a finite product of continuous images of Wasserstein groups. Therefore, \(O\) itself meets a finite product of continuous images of Wasserstein groups.

\smallskip

The Cameron--Martin affine group provides the nonsingular analogue of the Gaussian action considered by Glasner--Tsirelson--Weiss \cite[\S4]{glasner2005automorphism}:

\begin{exm}[Nonsingular Gaussian near action]\label{exm:CamMar}
Let \(\mathcal{H}=\ell_{\mathbb{R}}^{2}\). Consider the product space \(\mathbb{R}^{\infty}\) with its coordinate mappings \(\zeta_{j}:\mathbb{R}^{\infty}\to\mathbb{R}\), \(j\geq 1\). Let \(\gamma^{\infty}\) be the product of countably many copies of the standard Gaussian measure \(\gamma\) on \(\mathbb{R}\). There is a standard construction of Gaussian process \(\left(W_{h}:h\in\mathcal{H}\right)\) with coordinates in \(L^{2}\left(\gamma^{\infty}\right)\), that is determined by the covariances
\[\mathbb{E}_{\gamma^{\infty}}\left[W_{h}W_{k}\right]=\left\langle h,k\right\rangle_{\mathcal{H}},\quad h,k\in\mathcal{H}.\]
Then \(h\mapsto W_{h}\) identifies \(\mathcal{H}\) isometrically with the first chaos of \(L^{2}\left(\gamma^{\infty}\right)\), that is the closed linear span of \(\{\zeta_{j}:j\geq 1\}\); see \cite[Chs.~1--2]{Janson1997}. The usual probability preserving Gaussian near action of \(O\left(\mathcal{H}\right)\), considered by Glasner--Tsirelson--Weiss, is determined by the corresponding action on the first chaos: for \(U\in O\left(\mathcal{H}\right)\), let \(R_{U}\in\mathrm{Aut}\left(\mathbb{R}^{\infty},\gamma^{\infty}\right)\) be the automorphism class defined by
\[W_{h}\circ R_{U}=W_{U^{-1}h},\quad h\in\mathcal{H}.\]
This is the near action of the automorphism group of the Gaussian measure in \cite[\S4]{glasner2005automorphism}.

\smallskip

To define its nonsingular analog, for \(h\in\mathcal{H}\) let \(\alpha_{h}\) denote the Borel transformation of \(\mathbb{R}^{\infty}\) given by \(\alpha_{h}\left(x\right)=x+h\) with addition coordinate-wise. By the Cameron--Martin theorem, \(\alpha_{h}\) is nonsingular and
\[\frac{d\alpha_{h\ast}\gamma^{\infty}}{d\gamma^{\infty}}=\exp\big(W_{h}-\left\Vert h\right\Vert^{2}/2\big);\]
see \cite[Ch.~2, \S4]{Bogachev1998}. Therefore, the formula
\[\tau_{\left(h,U\right)}:=\alpha_{h}\circ R_{U},\quad \left(h,U\right)\in\mathcal{H}\rtimes O\left(\mathcal{H}\right),\]
defines a nonsingular near action \(\tau:\mathcal{H}\rtimes O\left(\mathcal{H}\right)\to\mathrm{Aut}^{\ast}\left(\mathbb{R}^{\infty},\gamma^{\infty}\right)\) of the Cameron--Martin affine group.
\end{exm}

\subsection{The isometry group of the Urysohn space}

Let \(\mathbb{U}\) be the Urysohn metric space, which is by definition the ultrahomogeneous universal complete separable metric space; see \cite[Ch.~5]{Pestov2006}. It is a Polish space, hence its isometry group \(\operatorname{Iso}\left(\mathbb{U}\right)\) is a Polish group with the pointwise convergence topology (see \cite[Prop.~5.2.1]{Pestov2006}). Vershik proved that \(\operatorname{Iso}\left(\mathbb{U}\right)\) contains a dense locally finite subgroup \cite[Thm.~5.3.8]{Pestov2006}, and Pestov upgraded this by showing that \(\mathrm{Iso}\left(\mathbb{U}\right)\) is a L\'{e}vy group using a skeleton of finite groups of isometries \cite[Thm.~5.3.10]{Pestov2006}, \cite{pestov2007isometry}. In the following we will show that \(\mathrm{Iso}\left(\mathbb{U}\right)\) is locally Wasserstein, using a result of Melleray--Tsankov.\footnote{The first arXiv version of this paper contained another proof, using Uspenskij’s \(g\)-embedding theorem \cite[Prop.~5.2.7]{Pestov2006}.} For the rest of this discussion set
\[G:=\operatorname{Iso}\left(\mathbb{U}\right).\]

\smallskip

Let \(\operatorname{Hom}\left(\mathbb{Z},G\right)\) be the space of homomorphisms \(\mathbb{Z}\to G\), equipped with the product Polish topology inherited from \(G^{\mathbb{Z}}\). By \cite[Lem.~6.2(i)]{melleray2013generic} applied to \(\mathbb{Z}\), the set
\[\big\{\pi\in\operatorname{Hom}\left(\mathbb{Z},G\right):\overline{\pi\left(\mathbb{Z}\right)}\cong L^{0}\left(\left[0,1\right],S^{1}\right)\big\}\]
is dense in \(\operatorname{Hom}\left(\mathbb{Z},G\right)\). Since the evaluation map \(\operatorname{Hom}\left(\mathbb{Z},G\right)\to G\), \(\pi\mapsto\pi\left(1\right)\), forms a homeomorphism of Polish spaces, this means that the set
\[\big\{g\in G:\overline{\left\{g^{n}:n\in\mathbb{Z}\right\}}\cong L^{0}\left(\left[0,1\right],S^{1}\right)\big\}\]
is dense in \(G\). Thus every nonempty open subset of \(G\) intersects a closed subgroup isomorphic to \(L^{0}\left(\left[0,1\right],S^{1}\right)\). By the second formulation in \cref{def:appWass}, \(G=\operatorname{Iso}\left(\mathbb{U}\right)\) is locally Wasserstein.

\subsection{General full groups}\label{sct:fullnonamen}

Let \(\mathcal{R}\) be a nonsingular countable Borel equivalence relation on a standard nonatomic probability space \(\left(X,\mu\right)\), and recall its full group \(\left[\mathcal{R}\right]\) as in \cref{sct:fullamen}. Here we no longer assume that \(\mathcal{R}\) is amenable, neither that it is ergodic. Therefore, according to \cite[Thm.~5.7]{giordano2007some}, \(\left[\mathcal{R}\right]\) may no longer be a L\'{e}vy group, and still we prove that it is a locally Wasserstein group.

\smallskip

To this end we use the three involution factorization of full groups due to Ryzhikov~\cite[p.~823]{ryzhikov1993factorization} (see also \cite[Prop.~2.6,\,Thm.~5.8]{miller2004full}), by which every \(T\in\left[\mathcal{R}\right]\) can be written as the composition \(T=O_{1}O_{2}O_{3}\) for \(O_{i}\in\left[\mathcal{R}\right]\) with \(O_{i}^{2}=\mathrm{Id}\), \(i=1,2,3\). Consequently, using the second formulation in \cref{def:appWass}, in order to show that \(\left[\mathcal{R}\right]\) is locally Wasserstein it suffices to show that for every involution \(O\in\left[\mathcal{R}\right]\) there exists a Wasserstein group \(H\) and a continuous homomorphism \(\iota:H\to\left[\mathcal{R}\right]\) such that \(O\in\operatorname{Im}\iota\).

\smallskip

Fix an involution \(\operatorname{Id}\neq O\in\left[\mathcal{R}\right]\), and find a Borel set \(B\subset X\) such that \(\operatorname{supp}\left(O\right)=B\sqcup O\left(B\right)\) (see \cite[Lem.~2.12]{hoareau2025polish}). Define a probability measure on \(B\) by
\[\beta\left(\cdot\right):=\frac{\mu\left(\cdot\right)+\mu\left(O\left(\cdot\right)\right)}{\mu\left(\mathrm{supp}\left(O\right)\right)},\]
and consider the Wasserstein group
\[H:=L^{0}\left(\left(B,\beta\right),\mathbb{Z}/2\mathbb{Z}\right).\]
Note that \(H\) is nothing but the measure algebra of \(\left(B,\beta\right)\): an element of \(H\) is identified with a Borel set \(C\subseteq B\) modulo \(\beta\)-null sets, and the product in the measure algebra, symmetric difference, corresponds to the sum modulo \(2\) in \(H\), in that \(\mathbf{1}_{C\triangle C'}=\mathbf{1}_{C}+\mathbf{1}_{C'}\mod 2\). Define
\[\iota:H\to\left[\mathcal{R}\right],\quad\iota\left(C\right)\left(x\right)=\begin{cases}
O\left(x\right) & x\in C\sqcup O\left(C\right)\\
x & \text{otherwise}
\end{cases}.\]
This is the flip by the involution \(O\) on the partial domain \(C\sqcup O\left(C\right)\). One verifies that
\[\iota\left(C\triangle C'\right)=\iota\left(C\right)\iota\left(C'\right),\]
showing that \(\iota\) is a group homomorphism, and that the uniform metric satisfies
\[d\left(\iota\left(C\right),\iota\left(C'\right)\right):=\mu\left(\iota\left(C\right)\neq\iota\left(C'\right)\right)=\mu\left(\left(C\triangle C'\right)\sqcup O\left(C\triangle C'\right)\right)=\mu\left(\operatorname{supp}\left(O\right)\right)\cdot\beta\left(C\triangle C'\right),\]
showing that \(\iota\) is continuous. Finally, we have that \(\iota\left(B\right)=O\), and so \(O\in\operatorname{Im}\iota\). This shows that \(\left[\mathcal{R}\right]\) is locally Wasserstein. Combining \cite[Thm.~5.7]{giordano2007some} with the above argument, we obtain the following:

\begin{cor}\label{cor:locWassnonLev}
Let \(\mathcal{R}\) be an ergodic nonsingular nonamenable countable Borel equivalence relation on a standard nonatomic probability space; e.g. the orbit equivalence relation of the Bernoulli shift
\[F_{2}\curvearrowright\big(\left\{ 0,1\right\}^{F_{2}},\bigotimes\nolimits_{F_{2}}\left(1/2,1/2\right)\big)\]
of the free group. Then \(\left[\mathcal{R}\right]\) is a locally Wasserstein group which is not a L\'{e}vy group.
\end{cor}

\section{A geometric obstruction to local compactness}\label[appendix]{app:geomlcsc}

Let \(G\) be a Polish group with a right-invariant compatible metric \(d\), and suppose it satisfies the condition of \cref{mthm2}: there is a sequence \(K_{1}\leq K_{2}\leq\dotsm\leq G\) of compact subgroups with dense union, each \(K_{n}\) admits a finite chain of compact subgroups
\[\{e\}=H_{n,0}\leq H_{n,1}\leq\dotsm\leq H_{n,M_{n}}=K_{n},\]
such that the quotient radii \(\rho_{n,i}=\sup_{g\in H_{n,i}}d\left(e,gH_{n,i-1}\right)\) satisfy
\[L:=\sup\nolimits_{n\geq 1}\sum\nolimits_{i=1}^{M_{n}}\rho_{n,i}<+\infty\quad\text{and}\quad\varrho_{n}:=\max_{1\leq i\leq M_{n}}\rho_{n,i}\to 0\text{ as }n\to\infty.\]

The conclusion of \cref{mthm2} says that \(G\) is a Wasserstein group and in particular, by \cref{mthm1}, it cannot be locally compact, unless it is the trivial group. We give a more direct argument to this conclusion. Thus, assume in addition that \(G\) is locally compact and we argue that \(G\) is trivial.

\begin{claim}
\(G\) is compact.
\end{claim}

Using that \(G\) is locally compact, choose \(\epsilon>0\) such that \(C:=\overline{B}_{d}\left(e,\epsilon\right)\) is compact. Let \(M:=\left\lceil 2L/\epsilon\right\rceil+1\), and choose \(n_{o}\geq 1\) such that \(\varrho_{n}<\epsilon/2\) for all \(n\geq n_{o}\). Divide the sequence \(M_{n},\dotsc,1\) into successive blocks \(b,\dotsc,a\) subject to the condition \(\rho_{n,b}+\dotsm+\rho_{n,a}\leq\epsilon\). There are at most \(M\) such blocks: otherwise, since each of the first \(M\) blocks contributes more than \(\epsilon/2\), the total sum is \(M\cdot\epsilon/2>L\), which is impossible.

\smallskip

Fix \(n\geq n_{o}\), and we show that \(K_{n}\subseteq C^{M}\). First, an arbitrary \(g\in K_{n}\) can be written as \(g=s_{M_{n}}\dotsm s_{1}\) with \(d\left(e,s_{i}\right)\leq\rho_{n,i}\) for \(1\leq i\leq M_{n}\). To do this, start with \(g_{M_{n}}:=g\). Since \(g_{M_{n}}\in H_{n,M_{n}}\) is \(\rho_{n,M_{n}}\)-distant from \(H_{n,M_{n}-1}\), there is \(h\in H_{n,M_{n}-1}\) with \(d\left(e,g_{M_{n}}h\right)\leq\rho_{n,M_{n}}\). Set \(s_{M_{n}}=g_{M_{n}}h\) and \(g_{M_{n}-1}:=h^{-1}\). Proceeding inductively, gives that \(g=s_{M_{n}}\dotsm s_{1}\) as required. Now for every block \(b,\dotsc,a\) within \(M_{n},\dotsc,1\) in the above division, by right-invariance and the triangle inequality
\[d\left(e,s_{b}\dotsm s_{a}\right)\leq d\left(e,s_{b}\right)+\dotsm+d\left(e,s_{a}\right)
\leq \rho_{n,b}+\dotsm+\rho_{n,a}\leq\epsilon,\]
which means that \(s_{b}\dotsm s_{a}\in C\). This shows that \(g\in C^{M}\), and so \(K_{n}\subseteq C^{M}\) for all \(n\geq n_{o}\). Since the \(K_{n}\)'s are increasing with dense union, it follows that \(G\subseteq C^{M}\), and hence \(G\) is compact.

\begin{claim}
\(m_{n}\Rightarrow m\) weakly, where \(m_{n}\) and \(m\) are the Haar probability measures of \(K_{n}\) and \(G\).
\end{claim}

Since \(G\) is compact, the family \(\left\{ m_{n}:n\geq1\right\}\) is relatively compact in the weak topology on the space of probability measures on \(G\). Every cluster point is \(\bigcup_{n\geq1}K_{n}\)-invariant, and by density it also \(G\)-invariant, so it is \(m\). Since this holds for an arbitrary cluster point, the claim follows.

\begin{claim}
Every \(1\)-Lipschitz function \(f:G\to\mathbb{R}\) is constant, and so \(G\) is trivial.
\end{claim}

Fix \(n\geq1\). For \(0\leq i\leq M_{n}\), let \(\mathcal{I}_{n,i}\) be the sub-\(\sigma\)-algebra of left \(H_{n,i}\)-invariant sets, and define
\[f_{n,i}:K_{n}\to\mathbb{R},\quad f_{n,i}:=\mathbb{E}_{m_{n}}\left[f\mid_{K_{n}}\mid\mathcal{I}_{n,i}\right].\]
Then \(f_{n,0}=f\mid_{K_{n}}\) and \(f_{n,M_{n}}=\mathbb{E}_{m_{n}}\left[f\right]\). Consider the martingale differences \(\Delta_{n,i}:=f_{n,i-1}-f_{n,i}\) for \(1\leq i\leq M_{n}\). As in the proof of \cref{lem:martcont}, these martingale differences are orthogonal in \(L^{2}\left(m_{n}\right)\). Additionally, by the first part of \cref{lem:martcont} applied to \(f\mid_{K_{n}}\),
\[\left\Vert\Delta_{n,i}\right\Vert_{L^{\infty}\left(m_{n}\right)}=\left\Vert f_{n,i-1}-f_{n,i}\right\Vert_{L^{\infty}\left(m_{n}\right)}
\leq\rho_{n,i},\quad1\leq i\leq M_{n}.\]
Then we obtain
\[\left\Vert f\mid_{K_{n}}-\mathbb{E}_{m_{n}}\left[f\right]\right\Vert_{L^{2}\left(m_{n}\right)}^{2}=\left\Vert f_{n,0}-f_{n,M_{n}}\right\Vert_{L^{2}\left(m_{n}\right)}^{2}=\sum\nolimits_{i=1}^{M_{n}}\left\Vert\Delta_{n,i}\right\Vert_{L^{2}\left(m_{n}\right)}^{2}\leq\sum\nolimits_{i=1}^{M_{n}}\rho_{n,i}^{2}\leq\varrho_{n}L\to 0\text{ as }n\to\infty.\]
Since \(m_{n}\Rightarrow m\) while \(f\) and \(f^{2}\) are continuous bounded, it follows that \(\left\Vert f-\mathbb{E}_{m}\left[f\right]\right\Vert_{L^{2}\left(m\right)}^{2}=0\) (since in general \(\left\Vert \phi-\mathbb{E}\left[\phi\right]\right\Vert _{2}^{2}=\mathbb{E}\left[\phi^{2}\right]-\mathbb{E}\left[\phi\right]^{2}\)). Thus an arbitrary \(1\)-Lipschitz function \(f\) is constant \(m\)-a.e. Finally, since in particular the \(1\)-Lipschitz function \(f\left(\cdot\right)=d\left(e,\cdot\right)\) is \(m\)-a.e. constant, \(G\) is trivial. 

\section*{Acknowledgments}

The authors are indebted to Zemer Kosloff for his advice and constant support, and to Emanuel Milman and Yair Shenfeld for sharing with us from their expertise. The authors are grateful to Fabien Hoareau and Emmanuel Roy, who independently suggested \cref{cor:locWassnonLev} to us.

N. A--R is thankful to Michael Bj\"{o}rklund for interesting conversations around the topics of this work. He also owes a great deal to Emmanuel Roy for introducing him to the work of Glasner--Tsirelson--Weiss, as well as for many insightful discussions on the Borel lifting problem.

This work is part of the Ph.D thesis of G.P. conducted at the Hebrew University under the supervision of Or Landesberg. G.P. would like to thank Or for his support and encouragement.

\bibliographystyle{acm}
\bibliography{References}

\end{document}